\documentclass[11pt]{amsart}
\usepackage{amssymb}
\usepackage{amsfonts}
\usepackage{color}
\usepackage{setspace}
\usepackage{xy}
\xyoption{all}

\newtheorem{theorem}{Theorem}[section]
\newtheorem{proposition}[theorem]{Proposition}
\newtheorem{lemma}[theorem]{Lemma}
\newtheorem{corollary}[theorem]{Corollary}
\newtheorem{definition}[theorem]{Definition}

\newtheorem{example}[theorem]{Example}

\newtheorem{conjecture}[theorem]{Conjecture}
\theoremstyle{remark}
\newtheorem{remark}[theorem]{Remark}

\renewenvironment{proof}{{\noindent\bf Proof.}}{\hfill $\Box$\par\vskip3mm}

\newcommand{\defqed}{\hspace*{\fill} $\square$}

\newcommand{\Hom}{{\rm Hom}}
\newcommand{\End}{{\rm End}}
\newcommand{\Ext}{{\rm Ext}}

\newcommand{\Aut}{{\rm Aut}\,}
\newcommand{\ann}{{\rm ann}\,}

\def\FF{{\mathbb F}}

\def\CC{{\mathbb C}}
\def\ZZ{{\mathbb Z}}
\def\QQ{{\mathbb Q}}

\def\KK{{\mathbb K}}

\begin{document}
\title{Incidence algebras and thin representation theory}

\begin{abstract}
We provide a unified approach, via deformations of incidence algebras, to several important types of representations with finiteness conditions, as well as the combinatorial algebras which produce them. We show that over finite dimensional algebras, representations with finitely many orbits, or finitely many invariant subspaces, or distributive coincide, and further coincide with thin modules in the acyclic case. Incidence algebras produce examples of such modules, and we show that algebras which are locally hereditary, and whose projective are distributive, or equivalently, which have finitely many ideals, are precisely the deformations of incidence algebras, and they are the finite dimensional algebra analogue of Pr${\rm \ddot{u}}$fer rings. New characterizations of incidence algebras are obtained, such as they are exactly algebras which have a faithful thin module. A main consequence of this is that ``every thin module comes from an incidence algebra": if $V$ is either thin, or $V$ is distributive and $A$ is acyclic, then $A/{\rm ann}(V)$ is an incidence algebra and $V$ can be presented as its defining representation. We classify thin/distributive modules, and respectively deformations, of incidence algebras in terms of first and second cohomology of the simplicial realization of the poset. As a main application we obtain a complete classification of thin modules over any finite dimensional algebra. Their moduli spaces are multilinear varieties, and we show that any multilinear variety can be obtained in this way. A few other applications, to Grothendieck rings of combinatorial algebras, to graphs and their incidence matrices, to linear algebra (tori actions on matrices), and to a positive answer to the ``no-gap conjecture" of Ringel and Bongartz, in the distributive case, are given. Other results in the literature are re-derived. 

\end{abstract}

\author{Miodrag C. Iovanov; Gerard D. Koffi}
\address{ 
Miodrag C. Iovanov\\
University of Iowa, 14 MacLean Hall\\
Iowa City, IA 52242-1419 }
\address{Gerard D. Koffi\\
University of Nebraska, Kearney,\\
Omaha, Nebraska
}

\email{miodrag-iovanov@uiowa.edu, yovanov@gmail.com;  koffigd@unk.edu; GerardKoffi@creighton.edu}

\date{}
\maketitle

\section{Introduction and Preliminaries}

In the representation theory of finite dimensional algebras, several types of finiteness conditions appear for representations, and algebras respectively; such are representations (algebras) with finitely many orbits from a (subgroup) of the units of the algebra, or finitely many submodules (ideals), distributive, thin. The intrinsic finiteness of these conditions make it so that combinatorialy defined algebras are a good source of examples producing such representations and algebras, and one main class of is that of incidence algebras.  

Incidence algebras of partially ordered finite (or more generally, locally finite) sets were introduced by G. Rota \cite{Ro} in the 60's in combinatorics to study inversion type formulas in a unified way. Given a finite poset (or quasi-ordered set) $P$, the incidence algebra $I(P,\KK)$ of $P$ over a field $\KK$ is the algebra with basis $f_{xy}$ indexed by pairs $x\leq y$ in $P$ and convolution (incidence) multiplication given by $f_{xy}f_{yz}=f_{xz}$, and $f_{xy}f_{zt}=0$ when $y\neq z$. These algebras have a well established history; they appear from many directions in mathematics and have thus been investigated by many authors,  from several (often overlapping) perspectives: combinatorial \cite{AFS, APRT, DS, Ro, Rv, SD, MW}, topological \cite{Cs,GR1,GS,IZ,W}, algebraic \cite{AA1, AA2, AHR1, AHR2, Ab, ABW,APRT,Ba,DIP,DK,F1,F2,GR1,GR2,Kr,simson,Sp0,Sp,SD}, representation theoretic \cite{Ba,DS,K1,K2,NR1,NR2,simson}, homological \cite{Cs,GR1,GR2,GS,IZ,Rd}; see also the monographs \cite{SD,W} and references within. This certainly represents only a modest list. Furthermore, incience algebras are special types of finite partial semigroup algebras, a theory which has seen great development during recent years \cite{Ho,O,Pu1,R}, and topological/homological methods have been fruitful and produced very interesting results in the representation theory of such structures (see \cite{MSS,MS} and references therein).
On the other hand, incidence algebras of finite type play a central role in representation theory of finite dimensional algebras, and implicitly appear, for example, in Nazarova and Roiter's work on the Brauer-Thrall conjectures, and Kleiner's work on subspace arrangements and finite type: any finite subspace arrangement can be completed with respect to intersections, and can be regarded as a special kind of representation of a suitable poset, where maps are all injective. Incidence algebras also provide a fertile ground for testing various conjectures, due to flexibility in choosing the poset, and in this respect have also been known as structural matrix algebras: that is, they can be defined as subalgebras of matrix algebras obtained by considering the set of matrices having arbitrary entries at a set of prescribed positions $(i,j)\in S$ and $0$ elsewhere. In particular, this definition automatically gives a representation of the incidence algebra which we will call 
the {\it defining representation}. 

An important property of incidence algebras is that they provide examples of distributive representations, that is, representations whose lattice of submodules (invariant subspaces) is a distributive lattice (i.e. $A\cap(B+C)=A\cap B+A\cap C$ for any submodules $A,B,C$). The class of distributive modules has also been studied by many authors (see for example \cite{Camillo, Stephenson} or the monograph \cite{T} and references within). Indeed, the projective indecomposable modules of an incidence algebra of a finite poset are distributive \cite{Ba,SD}. A module is called semidistributive if it is a direct sum of distributive modules. Hence, an incidence algebra is left and right semidistributive. Such modules show up frequently when studying algebras of finite representation type, that is, algebras which have only finitely many types of isomorphism of indecomposable modules up to isomorphism. More classically, on the commutative side, distributivity occurs naturally and examples play a central role: Dedekind domains, and more generally, Pr${\rm \ddot{u}}$fer domains are distributive rings. Also motivated by the interest in algebras of finite representation type, another kind of (representation of) an algebra or a ring $A$ which has been surfacing in literature is that for which the induced action of the group of units $U(A)$ (or some suitable subgroup) has finitely many orbits. 

The goal of this paper is to, one hand, completely classify distributive, thin and the other types of representations mentioned above, over any finite dimensional algebra, and on the other, fully understand their relation with the theory of incidence algebras and algebras with similar properties, and obtain consequences on the latter. Our main results will show that, in fact, all representations with such a finiteness condition come precisely from a combinatorial situation - an incidence algebra. Our unified approach is via a deformation theory of incidence algebras (which we introduce) and cohomology of the simplicial realization of the poset. Besides the aforementioned full classification of representations, which is one of our main results, many other applications are derived such as new characterizations and classifications of incidence algebras, their deformations, and of other algebras having a distributive, finiteness or other hereditary type properties, applications to graphs, incidence matrices, linear algebra and some open questions in representation theory, as well as consequences on representation and Grothendieck rings of such algebras; many results in the literature are re-derived as a consequence of the setup.


The paper is organized as follows: we first establish the general properties of representations with finiteness conditions in Section \ref{s.fin} and applications in Section \ref{s.3}. We then introduce the deformation theory, the main tool, in Section \ref{s.inc}, give classifications of classes of algebras with finiteness properties, and establish the relation to the representation theory of these algebras and some of the main results in Section \ref{s.rep}. We then give our applications to classifying thin representations Section \ref{s.thin}, to incidence matrices of graphs and linear algebra Section \ref{s.graphs}, to a partial answer to the ``no gap conjecture in Section \ref{s.ac} and to Grothendieck and representation rings in the last Section; we end with a list of questions.

\subsection{{\sc The results}}

Let $A$ be an algebra, $M$ an $A$-module. The properties of $M$, being distributive, or having finitely many $U(A)$-orbits are, as it turns out that, closely related to a third: namely, the condition that a representation has only finitely many invariant subspaces, again a ``finite type" condition. Hence, we first proceed to give a statement connecting the three conditions in general. Our first main result is the following (Theorem \ref{t1}): 

\begin{theorem}
Let $R$ be a ring, and $M$ an $A$-module. If $R$ is semilocal, then the following assertions are equivalent (and in general (1)$\Rightarrow$(2)$\Leftrightarrow (3)$). \\
(1) $M$ has finitely many orbits under the action of $U(R)$.\\
(2) $M$ has finitely many submodules.\\
(3) $M$ has finite length and any semisimple subquotient is squarefree of simple modules $T$ with infinite endomorphism ring; that is, if $T$ is a simple module with infinite endomorphism ring, then the module $T\oplus T$ is not a subquotient of $M$. \\
Moreover, when $A$ is an algebra over an infinite field, or an Artin algebra with no finite modules, then these statements are equivalent to $M$ being distributive.
\end{theorem}

The third condition in a modified form was already known to characterize distributivity \cite{Camillo, Stephenson}. Also, condition (2) was studied before by a few authors for the case of the regular representation (see \cite{Hirano} and references). In fact, the results of \cite{Hirano} characterizing rings $R$ which have finitely many orbits under the left action of $U(R)$ are re-obtained as a direct consequence of the previous theorem. 

Next, we aim to answer the question ``what is the general procedure to construct distributive representations?". As noted, incidence algebras produce many of such representations; one can then ask how general is this procedure on one hand, and on the other, whether having such distributive modules characterizes incidence algebras in some way. Feinberg \cite{F1,F2} has shown that algebras which have a faithful distributive representation, and satisfy two other technical conditions of a combinatorial nature, are incidence algebras. He also classified faithful distributive representations of incidence algebras by the group of multiplicative functions on the poset, a combinatorial notion. Still, one may ask whether a characterization which parallels that of quiver algebras of acyclic quivers (as finite dimensional hereditary algebras) to some extent exists. One important remark in this regard is that of Bautista, who introduced the locally hereditary algebras \cite{Ba} (algebras for which a local submodule of a projective module is projective) and showed that incidence algebras of finite posets are locally hereditary. Hence, one may ask to what extend does some converses hold.

To answer all these questions and put them in perspective, we introduce a class of algebras close to incidence algebras, and which are {\it flat deformations of incidence algebras}. This is somewhat close in spirit to homological methods and work present in other contexts in literature, including work related to Schurian algebras (i.e. algebras for which $\dim(eAf)\leq 1$ for all primitive orthogonal idempotents $e,f$); \cite{AA1,AA2,ACMT,IZ,MSS,MS}. The convergence of the algebra, topology and combinatorics will lead, however, to several classifications and descriptions of several classes of algebras and representations.

Briefly, we define a deformation $I_\lambda(P,\KK)$ of the incidence algebra of $P$ as the algebra obtained by simply twisting the multiplication of the combinatorial basis $f_{xy}$, by non-zero scalars $\lambda(x,y,z)$, such that it remains associative: the new multiplication $*_\lambda$ is then $f_{xy}*_\lambda f_{yz}=\lambda(x,y,z)f_{xz}, \,\lambda(x,y,z)\in\KK^*$ (Definition \ref{d.1}). We show that isomorphism types of such algebras are classified by the orbits of the action of the automorphism group of the poset $P$ on the singular cohomology $H^2(\Delta(P),\KK^\times)$, where $\Delta(P)$ is the simplicial realization of the poset $P$. Using this setup, the faithful distributive representations of the incidence algebra are easily seen to be in 1-1 correspondence with $H^1(\Delta(P),\KK)$, recovering in particular Feinberg's result \cite{F1}. By bringing in homological methods and reinterpretation, our result however can be used to directly compute the variety (set) of all such representations. 

The framework in all statements that follow is that of finite dimensional (basic) pointed algebras; that is, algebras for which every simple module is 1-dimensional (equivalently, $\End(S)=\KK$ for every simple $A$-module $S$, and $A$ is basic); up to Morita equivalence, one classically reduces to this case for representation theory. (Such algebras are usually just called pointed, but we may sometimes use basic pointed in order to emphasize the hypothesis).  

Our second main result is then the following structure and characterization statement (Theorem \ref{t.definc}), answering the above mentioned question regarding the distributivity and locally-hereditary properties, and further relating these to another finiteness (so combinatorial) condition - the finiteness of the set of two-sided ideals.

\begin{theorem}
Let A  be finite dimensional basic pointed $\KK$-algebra over an infinite field $\KK$. Then the following are equivalent. 
\begin{enumerate}
\item $A$ is a deformation of an incidence algebra of a finite poset (quasi-ordered set) $P$.
\item $A$ is locally hereditary and semidistributive (left, equivalently, right).
\item $A$ has finitely many two sided ideals and is (left or right) locally hereditary.
\end{enumerate}
\end{theorem}

Hence, one can view deformations of incidence algebras as the counterpart of Pr${\rm \ddot{u}}$fer rings in the world of finite dimensional algebras. We note that this theorem explains, on one hand, how far are algebras having the two properties characteristic to incidence algebras - (semi)distributivity and locally hereditary - from being an incidence algebra. It also shows that deformations of incidence algebras are very much like incidence algebras in many respects. On another hand, it provides a characterization of algebras having finitely many two sided ideals within the class of locally hereditary algebras. In fact, we show that in general, for the acyclic case (i.e. when the Ext quiver of $A$ is acyclic), an algebra $A$ has finitely many two sided ideals if and only if $A$ has finitely many left (equivalently right) ideals (Theorem \ref{t.twosided}), and hence, equivalent to the statement that $A$ is semidistributive.

One application of this is our third main result, which gives representation-theoretic characterizations of incidence algebras (Theorem \ref{t.inc}); recall that an $A$-representation $M$ is thin (or squarefree) if no simple $A$-module shows up as a factor in the Jordan-H$\rm\ddot{o}$lder series of $M$ more than once. By our results, in the acyclic case, a module is thin exactly when it is distributive, and this is equivalent to the condition that $M$ has only finitely many invariant subspaces.

\begin{theorem}
Let $A$ be a basic pointed algebra. The following are equivalent:
\begin{enumerate}
\item $A$ is an incidence algebra of a poset.
\item $A$ is acyclic and has a faithful distributive representation. 
\item $A$ has a faithful thin representation.
\item $A$ has a faithful representation with dimension (length) vector $(1,1,\dots,1)$ (i.e. all simples show up exactly once).
\end{enumerate}
\end{theorem}

This follows naturally from the cohomological setup, and is an application of general principles of deformation theory. First, use the representation theoretic conditions to get that an algebra $A$ with any of the properties in the Theorem above is locally hereditary and semidistributive; hence, it is a deformation of an incidence algebra by a 2-cocycle $\lambda$. The existence of a faithful distributive module $M$ translates to the fact that there is an element $\alpha$ in $H^1(\Delta(P),\KK^\times)$ which produces it. Then, interpreting the fact that $M$ is an $A$-representation in cohomology yields a connection between $\lambda$ and $\alpha$: namely, one gets that $\lambda$ is the image of the 1-cocycle $\alpha$ via the differential, so $\lambda$ is a coboundary. Thus, $\lambda\cong 1$ in $H^2$, the deformation is trivial, and $A$ is (isomorphic to) an incidence algebra. 

\subsection{Applications} \,

\vspace{.2cm}

{\it \noindent 1. Generic characterization of thin modules.} 

\vspace{.2cm}

One of the major applications of the results is the following ``generic" characterization of distributive modules, and of thin modules. Let $A$ be an algebra and $M$ an $A$-module. Assume that either $M$ is thin, or $M$ is distributive and $A$ is acyclic. Then, by the previous theorem, it follows that $A/\ann(M)$ is an incidence algebra of a poset $P$ (or quasi-ordered set if $A$ is not basic). Furthermore, after a choice of an incidence (combinatorial) basis of $A$ and a basis of $M$, we show that the representation can be seen precisely as (i.e. it reduces to) the defining representation of $A$. This shows that in the acyclic case, {\bf incidence algebras produce all possible distributive representations}, and in general they {\bf produce all possible thin representations}. This also shows that for any thin representation $M$ over any finite dimensional algebra $A$, a complete invariant for $M$ is given by a pair $(I,\alpha)$ where $I=\ann(M)$ is an ideal of $A$, and an element $\alpha\in H^1(P_I,\KK^\times)$, where $P_I$ is the poset for which $A/I\cong I(P_I)$. 

As another application, we also re-derive the results of \cite{F2} in the finite case. Furthermore, the same cohomological computations used to derive the above imply characterizations of automorphisms of deformations of incidence algebras, and in particular, on automorphisms of incidence algebras obtained by many authors \cite{Bc, DK, Kr, Sp0, Sp, S}. Using the cohomological/topological setup, we provide many examples to illustrate the results.




\vspace{.2cm}

{\it \noindent 2. Representations of incidence algebras.}

\vspace{.2cm}

As another application, we classify the all  thin representations (not necessarily faithful) of an incidence algebra $I(P,\KK)$. This is done in terms of the support (equivalently, annihilator) of such a representation, that is, a subposet $S$ of $P$ over which the representation becomes faithful. Such subposets $S$ of $P$ are precisely those which are closed under subintervals in $P$. The thin representations are then classified by and are in one-to-one correspondence with pairs consisting of closed a subposet $S$, and an element in the cohomology group $H^1(\Delta(S),\KK^\times)$, where $\Delta(S)$ is the topological realization of $S$ (Theorem \ref{t.discls}). 

Since the representations of the incidence algebra admit a monoidal structure (tensor product), we investigate this structure; the set of thin representations ${\mathcal D}$ is closed under tensor products, and we compute these products explicitly in terms of the support and multiplication on $H^1$ (Theorem \ref{t.RRdis}); the semigroup algebra $\ZZ[{\mathcal D}]$ is a particularly interesting and tractable subring of the representation ring. In the case when $P$ has meets $\wedge$ (infimum of any two elements), the Grothendieck group of $I(P,\KK)$ is also a ring, whose structure we compute and show that it is isomorphic to the semigroup algebra $\ZZ[P,\wedge]$. In particular, this provides a natural categorification of any commutative idempotent semigroup as a Grothendieck ring.   

We also note an interpretation of $H^3(\Delta(P),\KK^*)$ as parameterizing deformations of the monoidal category of vector spaces graded by (the partial semigroup of) intervals of $P$, analogue to (deformations) of categories of vector spaces graded by groups or monoids \cite{EGNO}, and we classify these via orbits of $\Aut(P)$. 

\vspace{.2cm}

{\it \noindent 3. Thin representation theory.}

\vspace{.2cm}

One of the main goals in representation theory is to classify interesting classes of representations. We use the methods here to give a complete classification of thin representations over an any arbitrary finite dimensional algebra $A$; this is done in Section \ref{s.thin}. This may well exist in the literature in some form, but we could not locate a reference, and we include it as it is closely related to the previous results and other linear algebra applications we obtain. This is one of our main applications. The classification is done first for nilpotent thin representations of an arbitrary quiver $Q$. In that case, the representations are again completely described by their support - an acyclic subquiver quiver $\Gamma$ of $Q$, and an element in $H^1(\Gamma,\KK^\times)$; they can be uniquely described by choosing a spanning tree $T$ of such quivers $\Gamma$, and placing arbitrary non-zero scalars for arrows in $\Gamma\setminus T$, and zero elsewhere (Theorem \ref{t.thin}). Using this, we derive a simple method to {\bf classify the thin representations of any finite dimensional algebra}, and obtain varieties giving 1-1 parametrizations of thin modules; they will be multilinear varieties. Furthermore we show that any such affine variety defined by multilinear polynomials can be obtained as a suitable variety of thin modules of a finite dimensional algebra.   


\vspace{.2cm}

{\it \noindent 4. Graphs and linear algebra.}

\vspace{.2cm}

In Section \ref{s.graphs}, we consider the following closely related general linear algebra problem: classify (parametrize) the orbits of the action by conjugation of the diagonal subgroup (torus) $Diag_n=Diag_n(\KK)$ of $Gl_n(\KK)$ on $M_n(\KK)$. This is quite possibly known, but, after also discussing with several specialists, we were unable to locate any reference. This is closely related to and may be of interest for problems of graph theory, and labeled graphs; it can be regarded as a ``canonical form problem" for conjugation by (invertible) diagonal matrices, and appears in some applications \cite{CI,gil}. It can also be interpreted as a question about certain thin representations of quivers, as it will be shown. Also, actions of tori are an active subject in itself, with many interesting questions, such as the {\it linearization problem}, which is worth mentioning; one the main open cases of this problem asks whether any algebraic group action of a torus on an affine space is linearizable (equivalent to a linear group representation).  

One associates naturally to any $n\times n$ matrix $A$ a quiver $\Gamma_A$ on $n$ vertices which has an arrow from $i$ to $j$ exactly when $a_{ij}\neq 0$; this is invariant under this action of $Diag_n$. Then, the orbits of this action are parametrized by such oriented graphs on $n$ vertices with each arrow of multiplicity at most 1 (Schurian quivers), together with an element in $H^1(\Gamma_A,\KK^\times)$. We also give an algorithm for finding such a canonical form for conjugation by invertible diagonal matrices. Hence, the answer is very closely related to the other results of the paper; in fact, considering the representation $V_A$ of $\Gamma_A$ naturally given by $A$ by placing the vector space $V_i=\KK$ at each vertex and using $a_{ij}$ for the morphism $V_i\rightarrow V_j$, one sees that the quotient $\KK[\Gamma_A]/{\rm ann}(V)$ of the quiver algebra of $\Gamma_A$ by the annihilator of $V_A$ is isomorphic to the incidence algebra of the quasi ordered set $(\{1,\dots,n\},\preceq)$, where $i\preceq j$ if there is an oriented path from $i$ to $j$, and hence $V_A$ is a thin (and distributive) representation of $\Gamma_A$ (but this is not nilpotent, in general). Moreover, as noted before, any thin representation, and any distributive representation of an acyclic algebra, arises this way.


As a consequence of these results, one may view distributive and thin representations, and the distributive/thin representation theory as the ``combinatorial" part of the representation theory of an algebra $A$. 

\vspace{.2cm}

{\it \noindent 5. The no gap conjecture in representation theory.}

\vspace{.2cm}

As another application, in Section \ref{s.ac}, we give an answer, in a particular case (the ``combinatorial case"), to an interesting conjecture due to Bongartz and Ringel. We recall that accessible modules over an algebra are defined inductively as follows (see \cite{Ri2}). First, simple modules are accessible; then, an indecomposable module $M$ of length $n$ is accessible if it has either an indecomposable submodule or an indecomposable quotient of length $n-1$. Recently, K. Bongartz \cite{Bo} proved the following remarkable result, which can be regarded as a Brauer-Thrall 0 type of result: if a finite dimensional (pointed) algebra $A$ admits a module of length $n$, then it admits indecomposable modules in every length $<n$; with a modification of Bongartz's argument, C. Ringel \cite{Ri2} proved the strengthened result that if $A$ has an indecomposable of length $n$, then it has an accessible module of length $n$. Hence, this motivates a stronger version of this statement as a  conjecture, which we formally list:

\begin{conjecture}[accessibility/no gap conjecture of Bongartz and Ringel]
Over any algebra, any indecomposable module of finite length is accessible.
\end{conjecture}

Hence, the conjecture asserts that the modules realizing the statement in Bongartz's above mentioned result could be chosen by taking successive maximal sub-quotients. We should point out that the author of \cite{Bo} comments that in general this conjecture is ``probably false". The proof of the result of \cite{Bo} is quite deep, and it has two parts: in the first (and shorter) part, the problem is reduced to distributive algebras, and the second considerably more involved part deals with the case of distributive algebras (algebras with finitely many ideals), which can be regarded as the ``combinatorial" part, and is based, in part, on covering theory techniques of Bongartz and Gabriel.

We give here an affirmative answer to the combinatorial case of this conjecture in Theorem \ref{t.accessible}:  we show that if $M$ is an indecomposable {\it distributive} representation over an acyclic finite dimensional basic pointed algebra $A$, of if $M$ is a thin module over an arbitrary algebra, then $M$ is accessible. Our setup of incidence algebras is suitable only for this, and one may possibly be able to avoid it altogether, but we hope that the underlying combinatorics will provide ideas as to how to go to the general (non-acyclic) case, and perhaps the general case for this conjecture. We give such a possible approach by considering ``thin covers" of arbitrary representations over any algebra.  

\subsection{Notations and conventions}

Throughout $P$ will be a {\it finite} poset set (partially ordered set). $\KK$ will be used to denote an infinite field, and $\KK$-algebras will always be finite dimensional, and Schur, that is, $\End(S)=\KK$ for every simple module $S$. By Morita equivalence, we reduce to the case when the algebra is basic (and hence basic pointed); the dictionary is that incidence algebras over posets correspond to basic structural algebras, and incidence algebras over quasi-ordered sets (posets without the symmetry axiom) correspond to arbitrary structural algebras. For $A$-modules $M,N$, the notation $\Hom(M,N)$ will always mean $\Hom_A(M,N)$ unless otherwise stated. We use the standard language of modules and representations interchangeably, and assume basic facts on modules and representation theory found in many textbooks \cite{AF, ASS, Cohn, F, Rotman}, to which we refer. 

In the interest of readability and to avoid a more technical language, as well as because of the connections to representation theory of finite dimensional algebras, we consider here only the case of finite posets $P$, and finite dimensional (incidence) algebras. We note, however, that most of the results that follow can be adapted for arbitrary locally finite posets (or quasi-ordered sets). This involves the use of the incidence coalgebra $\KK P$ (see e.g. \cite{simson}); the incidence algebra in the general case is simply the dual of the incidence coalgebra, and carries a natural topology, the finite topology. All the generalized statements will have a suitable formulation for the topological algebra $I(P,\KK)$ and its category of complete topological modules (also called pseudo-compact modules in the sense of Gabriel; see \cite{DNR}), or equivalently, they can be formulated for incidence coalgebras and their comodules. The general principle is than any comodule is a locally finite object (and {\it rational} module over $I(P,\KK)$), i.e. it is a sum of its finite dimensional subobjects. For coalgebras, the formulation of the results will involve the dual notions (injectives instead of projectives, locally hereditary defined dually with injectives, etc.). The deformations can be defined by analogy; the condition that an object has only finitely many subobjects will make sense to be considered only for finite dimensional comodules (finite dimensional rational $I(P,\KK)$-modules). 

Thus, theorems \ref{t.clsdef}, \ref{t.definc}, \ref{t.discls}, \ref{t.RRdis}, \ref{t.inc} will have corresponding analogues. For example, a deformation of an incidence algebra would be characterized as a (left, equivalently, right) locally hereditary and semidistributive coalgebra (of arbitrary dimension). The condition on finitely many ideals translates into the ``local discrete" condition ``any finite dimensional subcoalgebra has only finitely many subcoalgebras". The classification of deformations via orbits on cohomology, distributive representations, and considerations on categorical structures carry over accordingly. The faithful condition of Theorem \ref{t.inc} translates either to the existence of a complete topological faithful distributive module, or to the existence a faithful (or rather, cofaithful) distributive rational module. The results can then apply to recover the general statements in literature. 
We leave these extensions to the interested reader, or to a future follow-up.


\section{Modules with finitely many orbits}\label{s.fin}

\noindent The purpose of this section is to prove a general structure theorem for modules with finitely many orbits under the action of the group of units of the ring. This theorem is a generalization of the structure theorem for rings $R$ having finitely many orbits under the left regular action of $U(R)$, which is the main result of \cite{Hirano}. To fix terminology, we introduce

\begin{definition}
We say that an $R$-module as finitely many orbits if the action of the group of units $U(R)$ of $R$ on $M$ has finitely many orbits.
\end{definition}

The main theorem of this section is the following.

\begin{theorem}\label{t1} 
Let $R$ be a ring, and consider the action of the group of units $U(R)$ on $M$. Consider the following conditions for a  module $M$.
\begin{enumerate}
\item $M$ has finitely many orbits.
\item $M$ has finitely many submodules.
\item $M$ has finite length and has no subfactor isomorphic to $T\oplus T$, where $T$ is a simple left $R$-module with infinite endomorphism ring $\End_R(T)$.
\end{enumerate}
\noindent Then the following implications hold: $(1)\Rightarrow (2)\Leftrightarrow (3)$. Moreover, if $R$ is a semilocal ring, then $(3)\Rightarrow (1)$.
\end{theorem}

We note there is a close connection with a well studied and important notion in literature, namely, distributive modules. Distributivity is, of course, a classical subject that goes back to Dedekind and Pr$\rm\ddot{u}$fer domains (recall that one characterization of Pr$\rm\ddot{u}$fer domains is that they are distributive commutative domains); distributive modules have been extensively studied in the early 1970's (\cite[Chapter 3]{Cohn}, \cite{Stephenson}, \cite{Camillo}) and several characterizations and many results have been found for such modules; we mention the survey \cite{MT} and textbook \cite{T} and the many references therein. We provide here two apparently new results that describe arbitrary distributive modules.

The connection to our theorem is given by the result of V. Camillo \cite{Camillo} showing that a module $M$ is distributive if and only if all its semisimple subqotients are squarefree, that is, $M$ does not contain a subquotient isomorphic to $T\oplus T$ for some simple module $T$. In general, we say that a finite dimensional module (representation) $M$ is {\it squarefree} or {\it thin} if the multiplicity $[M:S]$ of every simple module $S$ in the composition series of $M$ is at most $1$. While Camillo's result holds in general, for convenience and self-containment, we recall the reader that for a semilocal ring $R$ this is easily observed by localization: if $e_1,\dots,e_n$ is a set of primitive orthogonal idempotents, by the exactness of the localizations $M\rightarrow e_iM$ and since $M=\bigoplus e_iM$, it is straightforward to note that an $R$-module $M$ is distributive if and only if $e_iM$ are distributive over the rings $e_iRe_i$; this, in turn, is obviously equivalent to $e_iM$ being uniserial over the rings $e_iRe_i$ (because $e_iRe_i$ are local). This condition is readily translated to the squarefree condition on $M$.

Before proving Theorem \ref{t1}, we note a few immediate consequences. 

\begin{corollary}
Let $A$ be a $\KK$-algebra, where $\KK$ is an infinite field. Then a left $A$-module $M$ has finitely submodules if and only it is a distributive module; furthermore, if $A$ is semilocal, this is equivalent to $M$ having finitely many orbits.
\end{corollary}
\begin{proof}
Obvious, since in this case for every simple module $T$, the $\KK$-vector space $\End(T)$ is an infinite set, and condition Theorem \ref{t1}(3) is precisely the squarefree test for distributivity. 
\end{proof}


\begin{corollary}
Let $R$ be a left artinian ring and $M$ a module which has no nonzero finite (as a set) factor module. Then $M$ has finitely many orbits if and only if $M$ has finitely many submodules, if and only if it is left artinian and distributive.
\end{corollary}
\begin{proof}
First, we show that if $M$ is artinian with no nonzero {\it finite quotient} module, then $M$ has no nonzero finite {\it subquotient} either. Assume the contrary, and then we may assume there is a finite simple subfactor  $S$ of $M$. 
Let $\bigoplus\limits_{i=1}^k P_i\rightarrow M\rightarrow 0$ be an projective cover epimorphism with projective local (indecomposable) modules $P_i$. Since $S$ shows up as a subquotient in a Jordan-H$\rm\ddot{o}$lder series of $M$, it shows up as a subquotient of some $P_i$. Let $T=top(P_i)$; then obviously $T$ is a quotient of $M$ too (by the cover condition), so $T$ must be infinite. Since $[P_i:S]\geq 1$, there is a (possibly length 0) path $top(P_i)=T=T_1\rightarrow T_2\rightarrow \dots\rightarrow T_n=S$ in the Ext quiver of $R$ (this is fairly well known; see, for example, \cite[Lemma 1]{Far}). Since $T_1$ is infinite, $T_n$ is finite, there is some arrow that makes a transition from an infinite to a finite $T_i$, i.e. ${\rm Ext}^1(T_i,T_{i+1})\neq 0$ for $T_i$ infinite and $T_{i+1}$ finite. But this is not possible when $R$ is left artinian, as shown, for example, in \cite[Proposition 2.1]{Iovanov}. \\
Now, notice that since $R$ is left artinian, $R/Jac(R)$ is semisimple and so each simple module $S$ is a finite dimensional vector space over $\End(S)$; therefore, $S$ is finite if and only if $\End(S)$ is finite. Hence, the hypothesis implies  that any simple subquotient of $M$ has infinite endomorphism ring, and so, the condition in Theorem \ref{t1}(3) is here equivalent to distributivity. Since $R$ is also semilocal, the triple equivalence follows from the Theorem.  
\end{proof}


\begin{corollary}
Let $A$ be an Artin algebra, and $M$ an artinian $A$-module. Then $M$ decomposes as $M=M_f\oplus M_t$ with $M_f$ finite and $M_t$ such that all its simple subquotients are infinite. Moreover, the following are equivalent:\\
(i) $M$ has finitely many orbits.\\ 
(ii) $M_t$ has finitely many orbits.\\
(iii) $M$ has finitely many submodules.\\
(iv) $M_t$ has finitely many submodules.\\
(v) $M$ is a direct sum of a finite module and an artinian distributive module.
\end{corollary}
\begin{proof}
The decomposition of $M$ as in the statement follows from \cite[Section 2]{Iovanov}, as $A=A_f\oplus A_t$, a direct sum of a finite ring $A_f$ and an Artin algebra with only infinite simple modules. For the $A_t$-module $M_t$ the condition of the previous Corollary holds, and everything follows from there. 
\end{proof}

\subsection*{Proof of the theorem}
 
To prove  Theorem \ref{t1}, we need a few observations. 

\begin{definition}  Let $M$ be a left $R$-module and $N_1, N_2,\dots, N_n$ be submodules of $M$. We say that $M=N_1\cup N_2\cup\dots\cup  N_n$ is an efficient union \cite{Gottlieb} if none of the $N_i$ can be omitted from the union.
\end{definition}

\begin{remark}\label{r1}
(1) Let $M$ be an $R$-module with finitely many orbits $U(R)x_i$, $i=1,\dots,n$. Then  $M=\bigcup\limits_{i=1}^n U(R)x_i\subseteq \bigcup\limits Rx_i\subseteq M$, and this union of submodules can be made efficient by eliminating some of the $Rx_i$'s. Moreover, every submodule $N$ of $M$ is also a module with finitely many orbits, since $N$ is closed under the action of $U(R)$ (hence it is a union of orbits among $\{U(R)x_i\}_i$), and so is (obviously) $M/N$. In particular, this shows that $M$ has finitely many submodules (since only finitely many unions can be constructed with the sets $U(R)x_i$), and this proves (1)$\Rightarrow$(2) in Theorem \ref{t1}. Also, $M$ is artinian of finite length.  \\
(2) Hence, if $N$ be a submodule of $M$, then  $N$, $M/N$ and also any subquotient of $M$, have finitely many orbits, but the converse is not true as will be evident in what follows. 
\end{remark}



The next remark is a standard folklore fact pointed out to us by Victor Camillo, and we include it for convenience. 

\begin{remark}\label{r2}
(1) Let $R$ be a ring and let $A,B$ be two left $R$-module. Consider the left $R$-module $A\oplus B$ and fix $\sigma_A,\sigma_B$ and $\pi_A,\pi_B$ the canonical injections and projections. Then there is a one-to-one bijection between $\Hom(A,B)$ and the set $C(B)$ of complements of $B$ in $A\oplus B$ given as follows. First, each $X\in C(B)$ is given uniquely by a morphism $h:A\oplus B\rightarrow B$ which is a retract of $\sigma_B$ (so which splits the sequence $0\rightarrow B\stackrel{\sigma_B}{\longrightarrow} A\oplus B\stackrel{\pi_A}{\longrightarrow} A\rightarrow 0$) - this is realized by letting $X=\ker(h)\in C(B)$. Furthermore, any such morphism $h$ with $h\circ \sigma_B={\rm Id}_B$ (i.e. $h\vert_B ={\rm Id}_B$) is  uniquely determined by $h\vert_A=h\circ \sigma_A\in\Hom(A,B)$.\\
(2) This shows that $\Hom(A,B)$ is finite if and only if $B$ has finitely many complements in $A\oplus B$. In particular, if $T$ is any module with infinite endomorphism ring, then $T\oplus T$ (and $T^n$ for $n\geq 2$) has infinitely many submodules (when $T$ is simple this is an if and only if statement). 
\end{remark}

\noindent\textbf{Proof of  Theorem \ref{t1}.} $(1)\Rightarrow (2)$ by Remark \ref{r1}.

\noindent $(2)\Rightarrow (3)$ Suppose $M$ has finitely many submodules; obviously, $M$ has finite length. If $M$ has a subquotient isomorphic to $T\oplus T$ with $T$ simple and $\textrm{End}_R(T)$ infinite, then  by Remark \ref{r2}, $T\oplus T$ has infinitely many submodules; but lifting back to $M$ this would produce infinitely many submodules of $M$. 

\noindent $(3)\Rightarrow (2)$  
We proceed by induction on the length $l(M)$ of $M$; when $l(M)=1$ it is obvious. For an $R$-module $Y$ and submodule $X$ of $Y$, let $L(Y)$ and $L_X(Y)$ denote the lattice of submodules of $Y$ and of submodules of $Y$ containing $X$, respectively. Let $\Sigma$ be the socle of $M$. Then $L(M)=\{0\}\cup \bigcup\limits_{S\subseteq \Sigma;\, S{\rm-simple}}L_S(M)$. Since each $L_S(M)=L(M/S)$ is finite by induction, if $\Sigma\neq M$ then the proof is finished as $\Sigma$ will have only finitely many submodules. Otherwise, $M$ is semisimple and let $S\subsetneq M$ be simple ($l(M)>1$). Then $L(M)=L_S(M)\cup (L(M)\setminus L_S(M))$, $L_S(M)$ is finite, and every submodule of $M$ which does not contain $S$, is contained in a complement $H$ of $S$ by semisimplicity. Each such complement has finitely many submodules (by induction), and there are only finitely many complements since their number equals $\alpha=|\Hom_R(S,M/S)|=|\End(S)^{n-1}|$ where $n$ is the multiplicity of $S$ in $M$; when $n=1$, $\alpha=|\{0\}|=1$ and when $n\geq 2$, then $S\oplus S\subseteq M$, by the hypothesis $\End_R(S)$ can only be finite, so $\alpha<\infty$. Thus $M$ has finitely many submodules.

\noindent $(2)\Rightarrow (1)$ 
We note that over a semilocal ring if two elements generate the same submodule then they are in the same orbit of $U(R)$, which will end the proof. This is essentially an observation of H. Bass:  if $x,y\in M$ such that $Rx=Ry$, then there exist $a,b\in R$ such that $y=ax$ and $x=b y$. Thus $(1-ab)\in\textrm{ann}(y)$ where $\textrm{ann}(y)$, the left annihilator of $y$. This implies $1=ab+c$ for some  $c\in\textrm{ann}(y)$, so $Rb+\textrm{ann}(y)=R$. By Lemma 6.4 in \cite{Bass}, we have that $(b+\textrm{ann}(y))\cap U(R)\neq\emptyset$. Let $u\in (b+\textrm{ann}(y))\cap U(R)$. Then $(b-u)y=0$, so that $x=by=uy$, and hence, $U(R)x=U(R)y$. 
\defqed

\begin{remark} The observation of Bass in the last part of the proof above seems to exist in other places;  for example, it appears to have been rediscovered in \cite[Lemma 1 and Lemma 2]{OR}, in a context also related to finitely many orbits for certain actions. W. Nicholson (private communication) has used similar arguments in recent work. 
\end{remark}

In the last part of the proof of this Theorem, we use the fact that the ring $R$ is
semilocal. We give below an example to show that if this is not true, then having finitely many submodules does not necessary imply that $M$ has finitely many orbits under the regular action.

\begin{example} Consider $R=\mathbb{C}\langle x,y\rangle$. Then $R$ is not semilocal.  Now,  let $V$ be a 2-dimensional vector space regarded as column vectors, with the $R$-module structure with $x$ and $y$ acting as the  matrices $A=\left[\begin{array}{cc}
1&0\\
1&1
\end{array}\right]$ and 
$B=\left[\begin{array}{cc}
1&1\\
0&1
\end{array}\right]$, respectively.  Then $V$ is simple since $A$ and $B$ do not have a common eigenvector. The group of units of $R$ is $\mathbb{C}$. Thus, for every non-zero $v\in V$, we have that $\mathbb{C}v$ is an orbit of $V$ under the regular action, and so $V$ has infinitely many orbits.
\end{example}

\section{Applications}\label{s.3}

In this section we give two applications of general nature of the result of Theorem \ref{t1}.

\subsection{Rings with finitely many one sided ideals or orbits}

\noindent  We apply Theorem \ref{t1} to the case where our module is the regular module, and obtain a new short proof of the structure theorem for rings with finitely many orbits under the left regular action of the group of units.


\begin{proposition}
Let $R$ be a ring with finitely many left ideals. \\
(i) Suppose $T,S$ are non-isomorphic simple $R$-modules and $S$ is infinite. Then $\Ext^1_R(T,S)=0$.\\
(ii) Moreover, if $S$ is infinite and $T$ is finite, then $\Ext^1(S,T)=0$.
\end{proposition}
\begin{proof} 
Assume otherwise; then there exists an indecomposable module $V$ of length $2$ with socle $S$ and top $T$, and a projective cover epimorphism $\pi:P(T)\rightarrow V$. Since $T\not\cong S$, $R=P(T)\oplus P(S)\oplus X$ for some $X$; if $N=\ker(\pi)\oplus J(P(S))\oplus X$, then $R/N=V\oplus S$ and contains $S\oplus S$ as a submodule. Thus, by Theorem \ref{t1}, ${}_RR$ has infinitely many submodules, a contradiction.\\
(ii) Since $R$ is left artinian, this follows, for example, from \cite[Proposition 2.1]{Iovanov} (which says that if $R$ is left artinian, $\Ext^1(K,L)\neq 0$ for simple modules $L,K$ then $|K|\leq |L|$ or they are both finite).
\end{proof}
 
We can now give a short new proof of \cite[Theorem 2.4]{Hirano} characterizing when the action of the group of units of $R$ has finitely many orbits.

\begin{theorem}(Hirano)
The following statements are equivalent for a ring $R$:
\begin{enumerate}
\item $R$ has only finitely many orbits under the regular action (i.e. action of $U(R)$).
\item $R$ has only finitely many left ideals.
\item $R$ is the direct sum of finitely many left uniserial artinian rings and a finite ring.
\end{enumerate}
\end{theorem}

\begin{proof} $(1)\Rightarrow (2)$ and $(3)\Rightarrow (1)$ follows from Theorem \ref{t1}.\\
$(2)\Rightarrow (3)$. By the previous proposition, since the finite and infinite simples are $\Ext$ orthogonal, $R$ has a block decomposition $R=R_f\times R_I$ with $R_f$ a finite ring and $R_I$ having only infinite simple modules. Also by this proposition, when $S,T$ are infinite, $\Ext^1(S,T)=\Ext^1(T,S)=0$, so $R_I$ is a product of local artinian rings $R_S$. Now by condition (3) in Theorem \ref{t1}, each $R_S$ will have no subfactor isomorphic to $S\oplus S$, so $J(R_S)^n/J(R_S)^{n+1}\cong S{\,\rm or\,0}$ as left modules. This shows that $R_S$ is left uniserial. 
\end{proof}

\subsection{Finitely many two sided ideals}

It is natural to ask what is the connection between the finiteness of the lattice of left ideals, or at least of the finiteness of the lattice of left submodules of certain indecomposable submodules, and the finiteness of the lattice of two sided ideals of the ring. We consider, in what follows, a basic artinian ring $A$, and let $e_1,\dots,e_n$ be primitive orthogonal idempotents. Since every two sided ideal $I$ of $A$ is $I=\bigoplus\limits_{i=1}^n{Ie_i}$ and $Ie_i$ are left submodules of $Ae_i$, it is immediate to note that if $Ae_i$ has finitely many submodules for each $i$ then $A$ has finitely many ideals. Of course, the converse is not true, as we see next.

\begin{example}
Let $Q$ be the Toeplitz quiver $\xymatrix{a\ar@(dl,ul)[]^x & b \ar[l]^y }$ and $A$ the quotient algebra of the path algebra $\KK[Q]$ of $Q$ by the relation $x^2=0$ (multiplication follows composition), where $\KK$ is an infinite field. It has a basis $\{a,b,x,y,xy\}$. Note that the right projective indecomposables are $P_r(b)={\rm Span}_\KK\{b\}$, $P_r(a)={\rm Span}_\KK\{a,x,y,xy\}$ and $P_r(a)$ has socle spanned by $\{y,xy\}$ and isomorphic to $\KK b\oplus \KK b$, so the socle is not squarefree. Hence, $P_r(b)$ has infinitely many submodules, but $A$ has finitely many two-sided ideals. Indeed, it is not hard to see that every two-sided is either $0$ or else contains $xy$. These are in bijection with ideals of the quotient of $\KK[Q]$ modulo $J^2$, whose only proper non-zero two sided ideals are easily seen to be spanned, respectively, by one of the sets $\{a,x,y\}$, $\{b,x,y\}$, $\{x,y\}$, $\{b,y\}$, $\{x\}$. 
\end{example}

Nevertheless, as it usually happens, when $A$ is acyclic, things are much better. We make a few observations. As usual, denote $S_i$ the simple module (at the top) of $Ae_i$, $J=J(A)$, so $S_i\not\cong S_j$ when $i\neq j$. We restrict to algebras over infinite fields which is the important case for representation theory, although this can be done in more generality. Assume $A$ is Schur, i.e. $\End_A(S)=\KK$ for every simple $S$ (for example, $A$ when $A$ is basic). Then the multiplicity $[Ae_i:S_j]$ of $S_j$ in $Ae_i$ is equal to $\dim\Hom_A(Ae_j,Ae_i)=\dim(e_jAe_i)$. If $A$ is acyclic, i.e. the $\Ext$ quiver of $A$ is acyclic, then it is well known that $\End_A(Ae_i)=\KK$ for all $i$ (since $[Je_i:S_i]=0$, as otherwise, as before, we would find a non-trivial path from $S_i$ to $top(Ae_i)=S_i$). So we can prove the following converse. Assuming $A$ is basic is not a restriction: up to Morita equivalence, finiteness of the lattice of ideals is a Morita invariant, since ideals of $A$ are in bijection with Serre subcategories of $A{\rm-Mod}$. 

\begin{theorem}\label{t.twosided}
Let $A$ be a basic pointed $\KK$-algebra with finitely many two-sided ideals ($\KK$ infinite). If $A$ is acyclic, then projective indecomposable $A$-modules are distributive, equivalently, they have finitely many submodules (equivalently, finitely many orbits).
\end{theorem}
\begin{proof}
First note that $\dim\Ext^1(S_i,S_j)\leq 1$. For $i=j$ this is $0$, and if $\dim\Ext^1(S_i,S_j)\geq 2$, then writing $A$ as a quotient of a path algebra by an admissible ideal, we would find that the path algebra $B$ of the quiver $\xymatrix{\bullet \ar@/^/[r]\ar@/_/[r] & \bullet }$ is a quotient of $A$. This is easily seen to have infinitely many two-sided ideals (the annihilator of each type of 2 dimensional uniserial module will be a different ideal), and then so would $A$ by pullback. \\
Next, we prove that $[Ae_i:S_j]\leq 1$, which will end the proof by Theorem \ref{t1}. Proceed by induction on the number of simple modules. Let $i$ be a source of the Ext quiver $Q$ of $A$, so $Ae_i=S_i$ is projective; assume $i=1$ without loss of generality. We show that $[Ae_i:S_1]\leq 1$. Assume otherwise; then there is $V\subseteq Ae_i$, $V\cong S_1\oplus S_1$ (since $S_1$ is projective, it can only occur at the socle of $Ae_i$). Let $x\in V$, so $x=xe_i$ and $S=Ax\cong S_1$ is 1-dimensional ($A$ is basic, so a cyclic submodule of $V$ is simple). Let $I_S:=SA=\bigoplus\limits_jSAe_j=SAe_i\oplus\bigoplus\limits_{j\neq i}SAe_j$ be the two sided ideal generated by $S$. If $a\in A$, $xae_i=xe_i\cdot e_iae_i=\lambda xe_i=\lambda x$ for some $\lambda\in \KK$, since this right action of $e_iAe_i$ represents an endomorphism of $Ae_i$, and $\End_A(Ae_i)=\KK$. Hence, $SAe_i=S$, and $I_S=S\oplus\bigoplus\limits_{j\neq i}SAe_j$. This shows that if $x,y\in V$ are two elements such that $Ax\neq Ay$, then $I_{Ax}\neq I_{Ay}$ (their projections onto $Ae_i$ are different). Since $V\cong S_1\oplus S_1$ has infinitely many simple submodules $Ax$, this implies that $A$ has infinitely many ideals, a contradiction.\\
Hence, $[Ae_i:S_1]\leq 1$ for all $i$. Now, let $\Sigma$ be the $S_1$-socle of $A$; it is a two sided ideal. Then since $S_1$ is projective, $[A/\Sigma:S_1]=0$ and the algebra $A/\Sigma$ has only $n-1$ isomorphism types of simples, and by induction, we get that $[Ae_i/(Ae_i\cap \Sigma):S_j]\leq 1$ because it is easy to see that $Ae_i/Ae_i\cap \Sigma$ are projective over $A/\Sigma$; this inequality lifts back to $[Ae_i:S_j]\leq 1$. One can also see this by localizing at $e_2+\dots+e_n$, and applying the induction hypothesis for the algebra $(e_2+\dots+e_n)A(e_2+\dots+e_n)$, to get that $[Ae_i:S_j]\leq 1$ for $i,j\geq 2$ (obviously, $[Ae_1:S_i]=0$ for $i\neq 1$ since $Ae_1=S_1$). 
\end{proof}

\begin{remark}\label{r.semidistributive}
Recall that a module is called semidistributive if it is a direct sum of distributive modules. We note that for an algebra $A$, projective indecomposable left modules are distributive if and only if it is left semidistributive (meaning it is semidistributive as a left module). This is obvious, since if $A=\bigoplus\limits_i M_i$ with $M_i$ distributive, then each $M_i$ is a direct sum of projective indecomposables, each of which will have to be distributive since they are submodules of the $M_i$'s.
\end{remark}

Hence, we can reformulate the main result of this section as 

\begin{theorem}
Let $A$ be a basic pointed algebra. If $A$ is acyclic, then $A$ is semidistributive if and only if $A$ has finitely many (two-sided) ideals.
\end{theorem}


\section{Incidence algebras and their deformations}\label{s.inc}


\noindent Let $(P,\leq)$ be a finite poset. Recall that the incidence algebra $I(P,\KK)$ of $P$ over $\KK$, where $\KK$ is a field, is defined to be the set of all functions $f: P\times P\rightarrow \KK$ such that $f(x,y)=0$ if $x\not\leq y$ with operations given by 
\begin{align*}
(f+g)(x,y)&=f(x,y)+g(x,y),\\
(f\cdot g)(x,y)&=\sum_{x\leq z\leq y}f(x,z)\cdot g(z,y)\textrm{ and }\\
(r\cdot f)(x,y)&=r\cdot f(x,y),
\end{align*}
\noindent for $f,g\in I(P,\KK)$ with $r\in \KK$ and $x,y,z\in P$. 

\noindent If $x\leq y\in P$, let $f_{xy}\in I(P,\KK)$ denote the function defined by 
$$f_{xy}(a,b)=\begin{cases} 1&\textrm{ if } (a,b)=(x,y),\\ 0 & \textrm{ otherwise.}\end{cases}$$ 
\noindent When $x=y$, we will denote $f_{xx}$ by  $f_x$. The set $\{f_{xy}\in I(P,k)| x\leq y\in X\}$ is a $\KK$-basis for $I(P,\KK)$. In particular, if $g\in I(P,\KK)$, we can write
$g=\sum_{x\leq y}a_{xy}f_{xy}$ where $a_{xy}=g(x,y)$.

An alternative definition of incidence algebras is that of structural matrix algebras. These have been studied by a few authors (see for example \cite{ABW} and references therein). A structural matrix algebra is, by definition, a subspace of $M_n(\KK)$ consisting of matrices having $0$ at fixed positions $(i,j)$ in a set $I\subseteq \{1,\dots,n\}\times \{1,\dots,n\}$, which is furthermore a subalgebra. This is the case if and only if there is a quasi-ordered set $(Q=\{1,2,\dots,n\},\preceq)$ (i.e. $\preceq$ is reflexive, transitive but not necessarily symmetric) such that $I=\{(i,j)| i\preceq j \}$. This algebra $A$ is isomorphic to the incidence algebra of the quasi-ordered set $Q$, and if $P=Q/\sim$ is the poset associated to $Q$ by the equivalence relation $a\sim b$ if $a\preceq b$ and $b\preceq a$, then $A$ is Morita equivalent to $I(P,\KK)$. Given an incidence algebra $I(P,\KK)$ viewed as a structural matrix algebra, consider the action of $I(P,\KK)$ on column vectors $\KK^n$; we call this the {\it defining representation} of $I(P,\KK)$. Alternatively, denoting $(e_x)_{\in P}$ the canonical basis of $\KK^n$ ($|P|=n$), then the action is defined by $f_{xy}e_y=e_x$. Obviously, this is well defined up to isomorphism of representations. We will see later how this representation also appears canonically from a cohomology class. 

Yet another definition of incidence algebras, of combinatorial importance, is that they are exactly the algebras which are dual to an incidence coalgebra of a poset, with multiplication being convolution.




We introduce now a class of algebras whose representation theory is close to that of incidence algebras; these will be certain deformations of the usual incidence algebra.

\begin{definition}\label{d.1}
Let $\lambda:\{(x,y,z)|x,y,z\in P, \,x\leq y\leq z\}\rightarrow \KK^*$, and denote for short $\lambda_{xz}^y=\lambda(x,y,z)$. The algebra $I_{\lambda}(P,\KK)$ is defined as $I_{\lambda}(P,\KK)=I(P,\KK)$ as a vector space, but with the deformed multiplication $*_\lambda$ on $I(P,\KK)$ given by 
$$f_{xy}*_\lambda f_{ yz}=\lambda_{xz}^y\cdot f_{xz},$$
 and 
$$f_{xy}*_\lambda f_{tz}=0$$ when  $y\neq t.$
We call this a deformation of the incidence algebra $I(P,\KK)$.
\end{definition}

The above algebra can be regarded in the context of deformation theory, although it departs slightly from it. It is nevertheless the algebra of relevance for our representation theoretic purposes. Let $Q$ be the Hasse diagram (quiver) of $P$: recall it has vertices $Q_0=P$ and an arrow $a\rightarrow b$ in $Q_1$ whenever $a<b$ minimally, i.e. there is no $c\in P$ with $a<c<b$. Let $\KK[Q]$ be the path algebra of $Q$, with multiplication given by concatenation (note that this is opposite than the more common convention of multiplication by composition; this agrees with the combinatorial multiplication of the incidence algebra). That is, if the paths $p,q$ are such that $p$ ends where $q$ begins, then we write $pq$ for ``$p$ followed by $q$". The incidence algebra $I(P,\KK)$ is the quotient of $\KK[Q]$ by the admissible ideal $I$ generated by relations of the form $p=q$, where $p$ and $q$ begin and end at the same point $a\in P$. Hence, $I(P,\KK)=T_{\KK Q_0}(\KK Q_1)/I$ where $A=\KK Q_0$ is a commutative semisimple algebra and $\KK Q_1$ is a bimodule over $A$. Then $I_\lambda(P,\KK)$ can be regarded as a flat deformation, in the sense that $\dim(I_\lambda(P,\KK))=\dim(I(P,\KK))$. 

\subsection*{Cohomology and deformations}

As usual, the formalism controlling such deformations of algebras has to do with cohomology; it will be, nevertheless, a simplicial (equivalently, singular) cohomology rather than Hochschild cohomology. 

\noindent Recall that to every poset $P$, there is an associated abstract simplicial complex $\Delta(P)$ (simplicial set) or $|P|$ called the geometric realization of $P$ as follows (see \cite{W}). The $0$-simplices of $\Delta(P)$ are the elements of $P$ and the set of $n$-simplices of $\Delta(P)$ are the totally ordered subsets of $P$ (i.e., the chains of $P$): $C^n=\{(a_0,a_1,\dots,a_n) | a_i\in P, \, a_0\leq a_1\leq\dots \leq a_n\}$. The face maps are the usual obvious ones. We note that sometimes $C^n$ is taken to consist of the strictly increasing chains; up to homotopy, the two define the same simplicial set (and, in fact, the two topological spaces obtained are homeomorphic).
If $\ZZ C^n$ is the free abelian group with basis $C^n$ as usual, we get a chain complex 
\begin{equation}\label{e3}
\xymatrix{
\dots \ar[r] & \ZZ C^n \ar[r]^{\partial_n} & \ZZ C^{n-1} \ar[r] & \dots
}
\end{equation}
\noindent realizing singular homology, with boundary map given by 
$$\partial_{n}(s_0,s_1,s_2,\cdots,s_n)=\sum_{i=0}^n(-1)^i(s_0,s_1,\cdots, s_{i-1},s_{i+1},\cdots,s_n).$$
\noindent We dualize (\ref{e3}) with respect to the abelian (multiplicative) group $\KK^*$, and  obtain the cochain complex 
\begin{equation}\label{e4}
\xymatrix{\dots \ar[r] & E^{n-1} \ar[r]^{\delta^n} & E^n\ar[r] & \dots}
\end{equation}
\noindent where $E^n=\textrm{Hom}_{\mathbb{Z}}(\mathbb{Z}C^n,\KK^*)$, $\delta^n=(\partial_{n})^*$ and 
$$(\partial_n)^*(f)(s_0,s_1,s_2,\cdots,s_n)=\prod_{i=0}^nf(s_0,s_1,s_2,\cdots,s_{i-1},s_{i+1},\cdots,s_n)^{(-1)^i}.$$

We note that the unusual multiplicative notation in the last equation is due to the fact that we use the group $(\KK^*,\cdot)$. At this point, we note that this set-up has been considered before, except that with  coefficients $\KK$. It is a classical result of Gerstenhaber and Schack that the Hochschild cohomology of the incidence algebra $I(P,\KK)$ is isomorphic (as graded algebras) with the singular cohomology of $\Delta(P)$. The same type of cohomology with coefficients in $\KK$ was also used by \cite{Cs,IZ} to study homological properties of $I(P,\KK)$. We have the following proposition straightforward (and standard) to check computationally.

\begin{proposition}
The multiplication $*_\lambda$ of $I_{\lambda}(P,\KK)$ associative if and only if
\begin{equation}\label{e4.1}
\lambda_{x,z}^y\cdot \lambda_{x,t}^z=\lambda^y_{x,t}\cdot\lambda_{y,t}^z
\end{equation}
\noindent for any $x\leq y\leq z\leq t\in P$, that is, if and only if $\lambda=\lambda(\bullet,\bullet,\bullet)$ is a 2-cocycle: $\lambda\in \ker(\delta^3)$.
\end{proposition}

\begin{proposition}
If $\lambda^{-1}\mu\in {\rm Im}(\delta^2)$ (i.e. $\lambda$ and $\mu$ differ by a coboundary) then $I_\lambda(P,\KK)\cong I_\mu(P,\KK)$ as (possibly non-associative) algebras. 
\end{proposition}
\begin{proof}
Indeed, the condition translates as $\lambda_{xz}^y(\mu_{xz}^y)^{-1}=\alpha_{xy}\alpha_{yz} (\alpha_{xz})^{-1}$ for $\alpha\in{\rm Im}(\delta^2)$, so $\lambda_{xz}^y\alpha_{xz}=\mu_{xz}^y \alpha_{xy}\alpha_{yz}$. We simply define $\varphi:I_\lambda(P,\KK)\rightarrow I_\mu(P,\KK)$ by $\varphi(f_{xy})=\alpha_{xy}f_{xy}$ and observe that $\varphi(f_{xy}*_\lambda f_{yz})=\varphi(f_{xy})*_\mu \varphi(f_{yz})$. 
\end{proof}

\begin{remark}
Note that when $\lambda\in Z^2(E^{\bullet})=\ker(\delta^3)$, then substituting $y=z=t$ and respectively $x=y=z$ in (\ref{e4.1}), we obtain $\lambda(x,y,y)=\lambda(y,y,y)$ and $\lambda(x,x,x)=\lambda(x,x,t)$ for $x\leq y$ and $x\leq t$, i.e. $\lambda_{xy}^y=\lambda_{yy}^y$ and $\lambda_{xx}^x=\lambda_{xt}^x$. Given $\lambda\in Z^2(E^{\bullet})$, let 
$\alpha\in E^{1}$  be defined by $\alpha(x,y)=\alpha_{xy}=(\lambda_{xy}^x)^{-1}=(\lambda_{xy}^y)^{-1}$ for $x\leq y$. Let $\mu=\lambda \delta^2(\alpha)$; then $\lambda^{-1}\mu\in {\rm Im}(\delta^2)$ and $I_\lambda(P,\KK)\cong I_\mu(P,\KK)$. Moreover, $\mu_{xx}^x=\lambda_{xx}^x\cdot\frac{\alpha_{xx}\alpha_{xx}}{\alpha_{xx}}=1$, hence, 
\begin{equation}\label{e6}\mu_{xx}^x=\mu_{xy}^x=\mu_{xy}^y=\mu_{yy}^y\end{equation} 
for all $x\leq y$. Using this, one easily sees that $\sum\limits_{x\in P}f_{xx}$ becomes an identity element for $I_\mu(P,\KK)$. Thus, the algebra $I_\lambda(P,\KK)$ is unital as well, and the isomorphism of the previous proposition is of unital algebras.
\end{remark}

\subsection*{Basic representation theory of deformations}

As mentioned, the algebra $I_\lambda(P,\KK)$ enjoys properties similar to $I(P,\KK)$. As noted above, $\lambda$ can be assumed to satisfy equation (\ref{e6}), and in this case $f_{xx}$ become orthogonal idempotents. In fact, they form a system of primitive orthogonal idempotents as seen below; we include a brief proof for completion.

\begin{lemma}\label{l.reps}
Let $I_\lambda(P,\KK)$ be a deformation of an incidence algebra and assume (without loss of generality) $\lambda$ satisfies (\ref{e6}). Then \\
(i) $(f_{xx})_{x\in P}$ is a system of primitive orthogonal idempotents, \\
(ii) the left modules $P_y={\rm Span}_\KK\{f_{xy},\,x\leq y\}$ and right modules $P'_y={\rm Span}_\KK\{f_{yz}|\,y\leq z\}$ are the left and respectively right projective indecomposables, \\
(iii) The Jacobson radical is spanned by $f_{xy}$ for $x<y$.\\
(iv) $S_x$, the simple top of each $P_x$, is 1-dimensional, and $I_\lambda(P,\KK)$ is acyclic. \\
(v) $P_x$ (and $P_x'$) are distributive, equivalently, have finitely many submodules (or orbits when $\KK$ is infinite), and their lattice of submodules is the same with that of $P_x$ for the usual incidence algebra $I(P,\KK)$. \\
(iv) The submodules of $P_x$ are of the form $M_{S}={\rm Span}_\KK\{f_{yx}|y\leq z,\,z\in S\}$ for some set $S\subseteq \{z\in P| z\leq x\}$; equivalently, they are submodules generated by a set of elements $\{f_{zx}|z\in S\}\subset P_x$ corresponding to vertices $z\in S$.
\end{lemma}
\begin{proof}
$f_{xx}$ are easily seen to be primitive, since $\dim(f_{xx}I_\lambda(P,\KK)f_{xx})=1$. There is an obvious direct sum decomposition
\begin{eqnarray*}
I_\lambda(P,\KK) & = & \bigoplus\limits_{y\in P}P_y
\end{eqnarray*}
which is obviously a direct sum of left modules, and then (ii) and (iii) follow. If $Q$ is the Hasse quiver of $P$, then it is easy to note that $J/J^2=\bigoplus\limits_{x\rightarrow y\in Q_1}\KK f_{xy}$, and this shows that $\Ext^1(S_y,S_x)$ is either $0$ dimensional, or otherwise it 1-dimensional exactly when there is $x\rightarrow y\in Q_1$ so (iv) follows. For the last part, since $\Hom(P_x,P_y)=f_{xx}I_\lambda(P,\KK)f_{yy}=\KK f_{xy}$ is one dimensional, the multiplicity $[P_y:S_x]\leq 1$, and so the squarefree condition of distributivity is satisfied.\\
(vi) is known for incidence algebras, and follows directly by computation: if $M$ is some submodule of $P_x$, let $S:=\{z\in P| f_{zx}\in M\}$. Then it is easy to see that $M=M_S$.
\end{proof}

\subsection*{Classification of deformations}

To prove a converse and classify deformations, we need another structural remark on cohomology. Let $\sigma\in \Aut_\leq(P)=\Aut(P,\leq)$ be an order automorphism of $P$ (a poset automorphism). Then $\sigma$ permutes the sets of $n$-simplices $C^n$, and its linear extension on $\ZZ C^n$ is easily seen to commute with $\partial^n$. Hence, $\sigma$ induces isomorphisms $\sigma:H_n(\Delta(P))\rightarrow H_n(\Delta(P))$ in homology, as well as in cohomology $\sigma:H^n(\Delta(P),\KK^*)\rightarrow H^n(\Delta(P),\KK^*)$. Hence, $\Aut(P,\leq)$ acts on $H^{\bullet}(\Delta(P),\KK^*)$ (on the right). The classification of deformations is as follows.

\begin{theorem}\label{t.clsdef}
The isomorphism classes of deformations $I_\lambda(P,\KK)$ of the incidence algebra of the poset $(P,\leq)$ are in one-to-one correspondence with the space $$H^2(\Delta(P),\KK^*)/\Aut(P,\leq)$$ of orbits of the action of $\Aut(P,\leq)$ on $H^2(\Delta(P),\KK^*)$.
\end{theorem}
\begin{proof}
For $\sigma\in \Aut(P,\leq)$ and $\lambda\in H^2(\Delta(P),\KK^*)$ denote the above described action by $\lambda{}^\sigma=\lambda(\sigma(\bullet),\sigma(\bullet),\sigma(\bullet))$.  The if part is straightforward: if $\lambda{}^\sigma=\mu$ modulo $B^2$, we first get $I_{\lambda{}^\sigma}(P,\KK)\cong I_{\mu}(P,\KK)$, and then note that $f_{xy}\mapsto f_{\sigma^{-1}(x)\,\sigma^{-1}(y)}$ is an isomorphism $I_{\lambda}(P,\KK)\cong I_{\lambda{}^\sigma}(P,\KK)$.\\
Now let $\phi:I_\lambda(P,\KK)\rightarrow I_\mu(P,\KK)$ be an isomorphism, and to avoid confusion, let $e_{xy}$ denote the elements of the basis of the first algebra $I_\lambda(P,\KK)$; thus $e_{xy}*_\lambda e_{yz}=\lambda_{xz}^ye_{xz}$ for $x\leq y\leq z$. Also, write $*_\mu=\cdot$ for brevity. We may assume that $\lambda,\mu$ satisfy (\ref{e6}) and so the elements $e_{xx}\in I_\lambda(P,\KK)$ and $f_{xx}\in I_\mu(P,\KK)$ are primitive orthogonal idempotents. Let $h_{xy}=\phi(e_{xy})$. Since both $(f_{xx})_x$ and $(h_{xx})_x$ are primitive orthogonal idempotents in $I_\mu(P,\KK)$, it follows that there is an invertible $a\in I_\mu(P,\KK)$ and a permutation $\sigma$ of $P$ such that $a^{-1} h_{xx} a\cong f_{\sigma(x)\,\sigma(x)}$  
(for example by \cite[Lemma 10.3.6]{HGK}; $\sigma$ is obtained a permutation such that $I_\mu(P,\KK)h_{xx}=I_\mu(P,\KK)f_{\sigma(x)\,\sigma(x)}$, etc.).  Then 
$$f_{\sigma(x)\sigma(x)} I_\mu(P,\KK) f_{\sigma(y)\sigma(y)}=a^{-1}(h_{xx} I_\mu(P,\KK) h_{yy}) a$$
(since $a^{-1}I_\mu(P,\KK)a=I_\mu(P,\KK)$). Therefore, the left hand side of the above equation is nonzero if and only if the right hand side is so, and this shows $\sigma(x)\leq \sigma(y) \Leftrightarrow x\leq y$, so $\sigma\in \Aut(P,\leq)$. Let $\beta=\mu{}^\sigma$ and $g_{xy}=f_{\sigma(x)\sigma(y)}$, so \begin{equation}\label{e.8}
g_{xy}g_{yz}=\beta_{xz}^yg_{xz}.
\end{equation} Note that $$g_{xx}(a^{-1}h_{xy}a)g_{yy}=(a^{-1}h_{xx}a)(a^{-1}h_{xy}a)(a^{-1}h_{yy}a)=\beta_{xy}^x\beta_{xy}^ya^{-1}h_{xy}a=a^{-1}h_{xy}a$$ 
by (\ref{e6}). An easy computation (using that $g_{xy}g_{zt}=0$ if $y\neq t$ and that $g_{xy}$ is a basis) shows that $a^{-1}h_{xy}a$ is a non-zero multiple of $g_{xy}$ for $x\leq y$; write $a^{-1}h_{xy}a=\alpha_{xy}g_{xy}$, $\alpha_{xy}\in \KK^*$. Using this and formula (\ref{e.8}), the equation $\phi(e_{xy}*_\lambda e_{yz})=\phi(e_{xy})\phi(e_{yz})$, after conjugation by $a$, yields  $\lambda_{xz}^y\alpha_{xz}=\alpha_{xy}\alpha_{yz}\beta_{xz}^y$. This means $\lambda \beta^{-1}\in {\rm Im}(\delta^2)$. Hence, $\beta\lambda^{-1}=(\mu{}^\sigma )\lambda^{-1}=1$ in $H^2(\Delta(P),\KK^*)$, i.e. $\mu{}^\sigma =\lambda$ in $H^2(\Delta(P),\KK^*)$. 
\end{proof}

We note an example where the incidence algebra has non-trivial deformations.

\begin{example}
Consider the poset given by the following Hasse quiver; the two diagrams represent the same poset. The solid arrows in the first diagram are the arrows of the Hasse quiver, and the dotted arrows in diagram represent relations that are not minimal (so they are not arrows in the Hasse quiver).
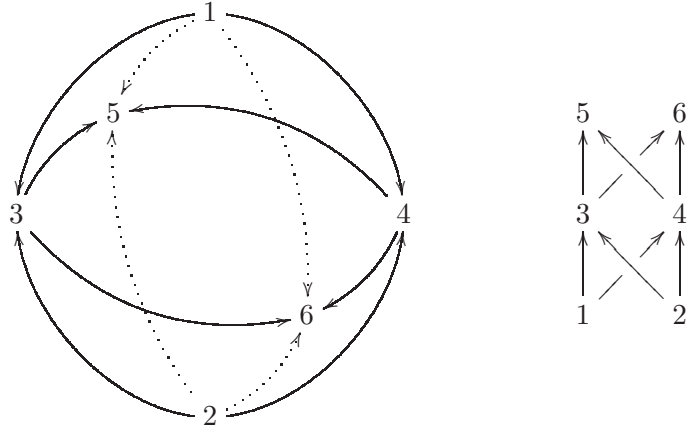
{$$\xymatrix{ & & 1\ar@(l,u)[ddll]\ar@(r,u)[ddrr] \ar@{..>}@/_{.5pc}/[dl] \ar@{..>}@/^{1pc}/[dddr]& &  & & & \\
& 5 &  & &   &  & 5 & 6 \\
3\ar@/^/[ur]\ar@/_{1.5pc}/[drrr] & & & & 4 \ar@/_{1.5pc}/[ulll]\ar@/^/[dl] & & 3 \ar[u]\ar[ur]|\hole & 4\ar[u]\ar[ul] \\
& &  & 6 &      & & 1\ar[u]\ar[ur]|\hole & 2\ar[u]\ar[ul] \\
& & 2\ar@(l,d)[uull]\ar@(r,d)[uurr] \ar@{..>}@/_{.5pc}/[ur]\ar@{..>}@/^{1pc}/[uuul]& &     & & &\\
}$$}
\noindent While the second diagram is more representative combinatorially, the first ``sphere" diagram shows visually that the simplicial realization of this poset is a 2-sphere; the dotted arrows belong to triangles which are the 2-simplices of this complex (there are 8 such 2-simplices). Hence, we have $H^2(\Delta(P),\KK^*)=\KK^*$ (for example, by universal coefficients). The automorphism group of this poset is easily seen to be $\ZZ/2\times \ZZ/2\times \ZZ/2$ (the pairs $(1,2)$, $(3,4)$, $(5,6)$ can be interchanged within themselves in any way, but no other permutation preserving the order can be constructed). It is finite and it has infinitely many orbits on $\KK^*$ ($\KK$ is assumed infinite; in fact, it can be shown the action is trivial). Thus, this poset algebra has infinitely many non-isomorphic deformations.
\end{example}

In fact, examples of non-trivial deformations are abundant; we will make more topological remarks in the next section when we look at representations of incidence algebras.

\subsection*{Automorphisms of deformations}

We make a brief remark about automorphisms of the deformations of incidence algebras. The automorphisms groups of incidence algebras have been studied by many authors \cite{Bc,DK,Sp0,Sp,S} (see also \cite{SD}), starting with Stanley's initiating paper \cite{S}. In general, there are three ``basic" types of automorphisms: \\
(1) the inner automorphisms ${\rm Inn}(I(P,\KK))$, which as a group is a quotient of $U(I(P,\KK))$ (units) by the central subgroup $(\KK^*)^t$, where $t$ is the number of connected components of $P$ (algebra blocks of $I(P,\KK)$); \\
(2) automorphisms induced by poset permutations $\Aut(P,\leq)$ (automorphisms of $(P\leq)$);\\  
(3) automorphisms determined by multiplicative functions, i.e. by $\alpha\in H^1(\Delta(P),\KK^*)$ ($\alpha_{xz}=\alpha_{xy}\alpha_{yz}$ for $x\leq y\leq z$) via $f_{xy}\mapsto \alpha_{xy}f_{xy}$. Then any automorphism of the incidence algebra is a product of three automorphisms, one of each type. We note that all these automorphisms remain automorphisms in the case of a deformation of an incidence algebra; indeed, taking a poset automorphism $\sigma\in \Aut(P,\leq)$ and inducing a linear map on $I_{\lambda}(P,\KK)$ amounts to acting with $\sigma$ on $\lambda\in H^2$, which does not change the isomorphism type (and hence the corresponding change of basis produces an automorphism); similarly, for $\alpha\in H^1(\Delta(P),\KK)$, the change of basis $f_{xy}\mapsto \alpha_{xy}f_{xy}$ amounts to changing $\lambda$ for $\delta^2(\alpha^{-1})\lambda$, preserving the isomorphism type again. Hence, the proof of Theorem \ref{t.clsdef} above in fact recovers the results on automorphisms of incidence algebras and can be interpreted as a statement on automorphisms of deformations: 

\begin{proposition}
Every automorphism of a deformation of an incidence algebras is a product of an inner automorphism, an automorphism induced by a  poset permutation $\sigma\in\Aut(P,\leq)$ (i.e. a poset automorphism) and one automorphism induced by some $1$-cocycle in $H^1(\Delta(P),\KK^*)$ as above. Moreover, the outer automorphisms group ${\rm Out}(I_\lambda(P,\KK))=\Aut(I_\lambda(P,\KK))/{\rm Inn}(I_\lambda(P,\KK))$ is isomorphic to $H^1(\Delta(P),\KK^*)\rtimes \Aut(P)$ just as it is in the non-deformed case, . 
\end{proposition}

The nature of the extension of the outer automorphism by the inner automorphism (creating the automorphism group) is related to the groups $Z^1(\Delta(P),\KK^*)$ and $B^1(\Delta(P),\KK^*)$ \cite{DK}. 

Similarly, one can extend results existing in literature on derivations of incidence algebras; see for example \cite{Bc}. Note that the outer derivations of an incidence algebra are described by $HH^1(I(P,\KK),I(P,\KK))$, which is isomorphic to $H^1(\Delta(P),\KK)$ by the Gerstenhaber-Schack isomorphism \cite{GS}, and any derivation of an incidence algebra is a sum of an inner one and one induced by an additive $\theta\in B^1(I(\Delta(P),\KK))$: $\theta(x,z)=\theta(x,y)+\theta(y,z)$ for $x\leq y\leq z$. It would be interesting to compute the Hochschild cohomology of deformations of incidence algebras; it is tempting to conjecture that $HH^*(I_\lambda(P,\KK),I_\lambda(P,\KK))=H^*(\Delta(P),\KK)=HH^*I((P,\KK),I(P,\KK))$ as the \cite{GS} technique and calculations seems to apply with a suitable modification (twist) by $\lambda$. 



\section{Representation theory of deformations of incidence algebras}\label{s.rep}

We start our considerations with the following easy remark. It shows that the distributivity property is in a way at the other end of the spectrum from the property of being faithful. First, we note that the condition $[P(S):S]=1$ for all simple $A$-modules $S$, is automatic when the algebra $A$ is acyclic; 
in fact, $[P(S):S]=1$ is equivalent to $P(S)$ being endotrivial $\End_A(P(S))=\KK$. We will come back to this type of condition later too.

\begin{proposition}\label{p.6.1}
Let $A$ a finite dimensional algebra and let $M$ be finite dimensional representation of $A$. The following hold. \\
(i) If $M$ is faithful, then $[M:S]\geq 1$ for all simple modules $S$. \\
(ii) Suppose either $A$ is acyclic, or more generally, $[P(S):S]=1$ for all simple modules $S$. Then $M$ is distributive if and only if $[M:S]\leq 1$ for all simple modules $S$. \\
Consequently, $M$ is a faithful distributive representation if and only if the multiplicity of every simple module in the series of $M$ is $1$. 
\end{proposition}
\begin{proof}
(i) is well known, as when $M$ is faithful, there is an embedding $A\hookrightarrow M^n$. The if part in (ii) is obvious. Suppose $M$ is distributive and $[M:S]\geq 2$ for some $S$ with $P(S)=Ae$. 
Hence, $\dim(eM)=\dim_A(\Hom(Ae,M))\geq 2$ and let $f,g:P(S)\rightarrow M$ be $A$-module morphisms with $f\neq \lambda g$ ($\lambda \in \KK^*$). Then ${\rm Im}(f)\neq {\rm Im}(g)$; otherwise, since $f:P(S)\rightarrow {\rm Im} f$ is a projective cover, we find $h\in \Aut(P(S))$ with $fh=g$, and since $\Aut(P)=\KK$, we'd get $f=\lambda g$, a contradiction. Let ${\rm Im}(f)=Ax$ and ${\rm Im}(g)=Ay$ for $x,y\in M$; these are local modules. \\
The idea is now to observe that the tops of $Ax$ and $Ay$ produce a 2-dimensional subquotient of $M$ isomorphic to $S\oplus S$.\\
Specifically, we show that $Ax\not\subseteq Ay+Jx$. Assume otherwise; then $x=ay+bx$ so $(1-b)x=ay$ for $b\in J=J(A)$; but then $1-b$ is invertible and this implies $x\in Ay$ so $Ax\subsetneq Ay$. If $P(S)\rightarrow Ay$ is a surjective morphism, then pulling back $Ax$ we get a (local) proper submodule of $P(S)$ with top isomorphic to $S$. Hence, $[JP(S):S]\geq 1$ and $[P(S):S]\geq 2$ a contradiction; 
similarly, we get that $Ay\not\subseteq Ax+Jy$. \\
Now, using this we see that $(Ax+Jy)/(Jx+Jy)$ and $(Jx+Ay)/(Jx+Jy)$ are non-zero submodules of $M/(Jx+Jy)$ and they are not equal submodules. Also, both are isomorphic to $S$ (since there is a surjective morphism $Ax/Jx\rightarrow (Ax+Jy)/(Jx+Jy)$ and this is bijective since $S$ is simple), so they form a direct sum inside $M/(Jx+Jy)$. Hence $M$ contains $S\oplus S$ as a subquotient and thus it is not distributive. This is a contradiction and the proof is finished.
\end{proof}


We now note the interpretation of the first cohomology group in representation theoretic terms. If $M$ is a faithful distributive representation of some acyclic basic pointed algebra $A$, then $[M:S]=1$ for every simple $S$; then $\dim(eM)=\dim(Ae,M)=1$. If $A=I_\lambda(P,\KK)$, then $f_{yy}M=\KK m_y$ (it is 1-dimensional spanned by some $m_y$). Then $f_{xy}\KK m_y=f_{xx}(f_{xy}\KK m_y)\subseteq f_{xx}M=\KK m_x$. Let $\alpha_{xy}$ be defined by $f_{xy}\cdot m_y=\alpha_{xy}m_x$ for $x\leq y$. Obviously, $f_{xy}\cdot m_z=0$ if $y\neq z$, and $(\alpha_{xy})_{x,y}$ determines the representation. Denote $M_\alpha$ the representation obtained for such a particular system of coefficients $\alpha\in E^1$ (it will need to satisfy a 1-cocycle condition as noted in what follows).


\begin{theorem}
If  $I_\lambda(P,\KK)$ has a faithful distributive representation, then $I_\lambda(P,\KK)\cong I(P,\KK)$ is a trivial deformation. Moreover, the faithful distributive representations of $I(P,\KK)$ are in 1-1 correspondence with $H^1(\Delta(P),\KK^*)$; more precisely, $M=M_\alpha$ is a representation exactly when $\alpha$ is a 1-cocycle, and $M_\alpha$ is uniquely determined up to isomorphism by the cohomology class $[\alpha]\in H^1(\Delta(P),\KK^*)$.
Under this bijection, the defining representation of $I(P,\KK)$ corresponds to the trivial 1-cocycle $[\alpha]=[1]$.
\end{theorem}
\begin{proof}
If $\alpha$ is as above, then writing the representation condition $(f_{xy}*_\lambda f_{yz})\cdot m_z=f_{xy}\cdot (f_{yz}\cdot m_z)$ yields $\lambda_{xz}^y\alpha_{xz}=\alpha_{xy}\alpha_{yz}$, which means $\lambda=\delta^2(\alpha)$, so $[\lambda]=[1]$ in $H^2(\Delta(P),\KK)$; this proves the first part. Furthermore, now picking $\lambda\equiv 1$ (constant function), a faithful distributive representation of $I(P,\KK)$ is given by $\alpha$ for which $\alpha_{xz}=\alpha_{xy}\alpha_{yz}$, i.e. $\alpha\in \ker\delta^2$. An isomorphism $\varphi:M_\alpha\rightarrow M_\beta$ of two such representations $M_\alpha,M_\beta$ given by $\alpha,\beta$ will induce $\varphi(eM_\alpha)=eM_\beta$ for any idempotent $e$, thus $\varphi(m_x)=\theta_x n_x$, $\theta_x\in \KK^*$. Hence, $\varphi(f_{xy}\cdot m_y)=f_{xy}\varphi(m_y)$ ($x\leq y$) implies $\alpha_{xy}\theta_x m_x=\beta_{xy}\theta_y m_x$, and so $\alpha_{xy}\beta_{xy}^{-1}=\theta_x^{-1}\theta_y$; the existence of such a $\theta$ is equivalent to $[\alpha]=[\beta]$ in $H^1(\Delta(P),\KK)$.\\
Finally, the last statement follows since the defining representation is constructed by $\alpha_{xy}=1$.
\end{proof}

We note that the above classification of faithful distributive representations is essentially identical to that of \cite{F1}, where the cohomology connection was not remarked (elements $\alpha\in Z^1$ were called multiplicative functions, a term preferred by combinatorialists). Our approach is direct, as it is an immediate consequence of the cohomological setup. Nevertheless, the geometric/topological intuition allows us to give plenty of non-trivial examples (as noted before, we chose to restrict to finite dimensional algebras). 

\begin{example}
Let $P$ be the poset whose Hasse diagram is 
$$
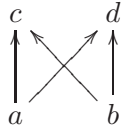
\xymatrix{ 
c  & d \\
a \ar[u]\ar[ur] & b\ar[ul]\ar[u]
}$$
Obviously, $\Delta(P)$ is homeomorphic to $S^1$ (as topological spaces). 
Hence, $H^1(
\Delta(P),\KK^*)=\KK^*$ and this classifies faithful distributive representations. Explicitly, any such representation is given by a system of 4 scalars $\alpha_{ac},\alpha_{ad},\alpha_{bc},\alpha_{bd}$ and the product $\alpha_{ac}\cdot\alpha_{bc}^{-1}\cdot\alpha_{bd}\cdot\alpha_{ad}^{-1}\in \KK^*=H^1(\Delta(P),\KK^*)$ represents a complete invariant of this representation (some scalars can be exchanged for their inverses so that the relation looks like $\alpha_{ac}\cdot\alpha_{bc}\cdot\alpha_{bd}\cdot\alpha_{ad}\in \KK^*=H^1(\Delta(P),\KK^*)$). 
\end{example}

\begin{example}\label{e.3}
More generally, consider the poset with associated Hasse diagram
$$
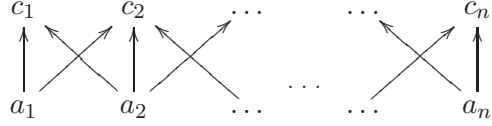
\xymatrix{ 
c_1  & c_2 & \dots & \dots & c_n\\
a_1 \ar[u]\ar[ur] & a_2\ar[ul]\ar[u]\ar[ur] & \dots\ar[ul]\ar@{}[ur]_(.3){.\,\,.\,\,.} & \dots\ar[ur] &  a_n\ar[ul]\ar[u]
}$$
Topologically this is a wedge of $n-1$ circles and hence the faithful distributive representations are parametrized by $H^1=(\KK^*)^{n-1}$.
\end{example}

\begin{remark} {\bf [Topological considerations]}\\
(1) In general, by universal coefficients, note that $H^1(\Delta(P),\KK^*)=\Hom(H_1(\Delta(P)),\KK^*)$ where $H_1(\Delta(P))=H_1(\Delta(P),\ZZ)$. Of course, $H_1$ is a finitely generated abelian group, and so $H^1$ decomposes in a direct product of copies of $\KK^*$ and $\Hom_\ZZ(\ZZ/p^n,\KK^*)$; this latter term depends on existence of roots of unity in $\KK$, which, in turn, depends on whether $\KK$ is algebraically closed or not and on the characteristic of $\KK$. For example, if $H_1(\Delta(P),\ZZ)=\ZZ/2\times \ZZ/p$ for an odd prime $p$, then $|H^1(\Delta(P),\QQ)|=2$, $|H^1(\Delta(P),\overline{\FF}_2)|=p$, $|H^1(\Delta(P),{\CC})|=2p$, so the number of faithful distributive representations depends on characteristic and algebraic closure properties, so that the same incidence algebra can have a different number of faithful indecomposable representations over different fields. \\
(2) Again, by universal coefficients $$H^2(\Delta(P),\KK^*)\cong \Hom_\ZZ(H_2(\Delta(P)),\KK^*)\oplus \Ext_\ZZ^1(H_1(\Delta(P)),\KK^*).$$
If $\KK$ is algebraically closed, then the group of units of $\KK^*$ is injective (direct sum of Pr$\rm\ddot{u}$fer groups $C_{p^\infty}$; when ${\rm char}(\KK)=q$ the group $C_{q^\infty}$ is missing but all other $C_{p^\infty}$ show up). Hence, in that case, $H^2(\Delta(P),\KK^*)\cong \Hom_\ZZ(H_2(\Delta(P)),\KK^*)$; otherwise, terms from the $\Ext$ part are possible as well. This means that it may be possible that $H_2(\Delta(P))=0$, but still the incidence algebra can have nontrivial deformations. \\
(3) Nevertheless, we note that {\it if $H_2(\Delta(P))$ is finite, then $I(P,\KK)$ has only finitely many deformations up to isomorphism} (in which case one can think of the incidence algebra as ``semi-rigid"). Indeed, in this case, the torsion part of $(\KK^*,\cdot)$ (i.e. the group of roots of unity) is a direct sum of a finite group and an injective abelian group with squarefree socle, and therefore  $\Hom_\ZZ(H_2(\Delta(P),\KK^*))$ is finite. Also, $\Ext_\ZZ^1(H_1(\Delta(P)),\KK^*)$ is finite, since only the torsion part of $H_1$ (which is finite, since $H_1(\Delta(P))$ is finitely generated) can contribute to this Ext group, which becomes an Ext of two finite abelian groups.\\
(4) We use a few well known classical topological facts to observe that the structure of the groups giving deformations and representations is quite arbitrary. Given any CW-complex which is a $\Delta$-complex, or more generally a simplicial set, one can take the second barycentric subdivision and obtain an abstract simplicial complex whose topological space (geometric realization) is homeomorphic to the realization of the original simplicial set (see for example, \cite[Section 2.1 and Appendix]{H}). Also, starting with an abstract simplicial complex $X$ (so each simplex is determined by its vertices), let $P=P(X)$ be the poset of its faces with inclusion as order; then the geometric realization $\Delta(P(X))$ is homeomorphic to $X$ \cite{W}. Since any (finite) sequence of finitely generated abelian groups can be realized as the (co)homology of some (finite) CW-complex, using also Quillen's well known equivalence of categories between simplicial sets and topological spaces up to homotopy, we see that the cohomology of $\Delta(P)$ is really arbitrary. This shows that there are posets in any of the above discussed situations, and with any prescribed number of deformations and distributive representations.
\end{remark}


We give now the main result on characterizations of deformations of incidence algebras. As noted, their projective indecomposables are distributive, and it this one of the  main properties characterizing this class of algebras. We recall another terminology, which is a natural extension of hereditary algebras: following Bautista, an algebra is said to be l-hereditary or locally hereditary \cite{Ba,Ri} if local submodules of projective indecomposable (equivalently projective) modules are projective. An algebra is l-hereditary if and only if every non-zero morphism between projective indecomposable (equivalently, local) modules is injective. 

\begin{remark}
We note that a locally hereditary algebra is necessarily acyclic. First, fix a decomposition $A=\bigoplus\limits_eAe$ and let $S_e=Ae/Je$. Note that if $\Ext^1(S_f,S_e)\neq 0$, then there is a length two indecomposable module $V$ and a short exact sequence $0\rightarrow S_e\rightarrow V\rightarrow S_f\rightarrow 0$, so there is a non-zero morphism $Ae\rightarrow Af$ which must be injective, but clearly not surjective. Hence, any arrow $S_f\rightarrow S_e$ in the Ext quiver corresponds to an (strict) embedding $A_e\hookrightarrow A_f$, and so an oriented cycle containing $S_e$ would produce a non-surjective embedding $Ae\hookrightarrow Ae$, impossible. 
\end{remark}

\begin{theorem}\label{t.definc}
Let A  be finite dimensional basic pointed $\KK$-algebra (with $\KK$ infinite). Then the following are equivalent. 
\begin{enumerate}
\item $A$ is a deformation of an incidence algebra.
\item $A$ is (left, equivalently, right) locally hereditary and semidistributive.
\item $A$ has finitely many two sided ideals and is (left or right) locally hereditary.
\item For every indecomposable projectives $P,Q,R$, $\dim\Hom(P,Q)\leq 1$ and the natural composition map $\Hom(P,Q)\times \Hom(Q,R)\rightarrow \Hom(P,R)$ is non-zero whenever the first two Hom spaces are non-zero.
\end{enumerate}
\end{theorem}
\begin{proof}
(1)$\Rightarrow$(2),(3),(4) This is similar to the incidence algebra case; as we have seen, $P_x=P(S_x)$ are distributive. Also, by the Lemma \ref{l.reps}, the submodules of $P_x$ are the $M_S$ generated by $f_{zx}$ for $z\in S$; we may assume the elements of $S$ are not comparable (since if $a,b\in S$ and $a\leq b$, it is immediate by the definition that $M=M_{S\setminus \{a\}}$). Let $S'=\{y\in P | \exists z\in S {\rm \,s.t.\,}y<z\}$; it is easy to see that  $M_S/M_{S'}\cong \bigoplus\limits_{z\in S}S_z$. This can be local only when $|S|=1$, so then $M_S=A f_{zx}$ for $S=\{z\}$. This cyclic and has a basis consisting of $\{f_{yx}|y\leq z\}$, and so $\dim(Af_{zx})=\dim(P_z)$, since $P_z$ has basis $\{f_{yz}|y\leq z\}$. The Jacobson radical of $Af_{zx}$ is obtained by multiplication with $f_{uv}$ for $u<v$, and so we see that the top of $Af_{zx}$ is simple isomorphic to $S_z=\KK f_{zz}$. Hence, there is a projective cover map $P_z\rightarrow Af_{zx}$, which is an isomorphism by dimension, so $Af_{zx}$ is projective.\\ 
(2)$\Leftrightarrow$(3) follows by Theorem \ref{t.twosided}.\\
(4)$\Rightarrow$(1) The reconstruction is an argument similar to that present in \cite{AA1,AA2} for such ``square-free" algebras (\cite{ACMT}). $\dim(\Hom(P,Q))\leq 1$ for all indecomposable projectives $P,Q$ is equivalent to $[P:S]\leq 1$ for all projective indecomposables $P$ and simple modules $S$, i.e. projective indecomposables are distributive (the trivial part of Proposition \ref{p.6.1}). Let $\mathcal{P}$ be a set of primitive orthogonal idempotents, so $\dim(eAf)\leq 1$ for all $e,f\in \mathcal{P}$. Define $e\leq f$ if and only if $eAf\neq 0$. The non-zero map in hypothesis means that $e,f,g\in \mathcal{P}$, $e\leq f$ and $f\leq g$ implies $e\leq g$. Thus $(\mathcal{P},\leq)$ is a poset, and now if $\{q_{ef}\}$ is a basis of $eAf$ if $e\leq f$ (so $eAf=\KK q_{ef}$ for $e\leq f$), one easily obtains that $q_{ef}\cdot q_{fg}=\lambda(e,f,g)q_{eg}$ for some $\lambda(e,f,g)\in\KK^*$ ($e\leq f\leq g$). This means that $A$ is a deformation of an incidence algebra (in particular, $A$ is acyclic).\\
(2)$\Rightarrow$(1) is similar: the semidistributive property implies $[P:S]\leq 1$ by Proposition \ref{p.6.1}, so $\dim(eAf)\leq 1$ for primitive idempotents $e,f$. Proceed as in (4)$\Rightarrow$(1) and note the locally hereditary condition means that maps between $Ae$'s are either 0 or injective, and hence the second condition in (4) follows. The proof then concludes as in the previous case.
\end{proof}

We can now give the consequences on characterizations of incidence algebras. As before, we say that a representation $M$ of an algebra is multiplicity-free if $[M:S]\leq 1$ for every simple $S$. We say that a representation is a 1-representation if $[M:S]=1$ for all simple modules $S$. 

\begin{theorem}\label{t.inc}
Let $A$ be a basic pointed algebra. The following are equivalent:
\begin{enumerate}
\item $A$ is an incidence algebra of a poset.
\item $A$ is acyclic and has a faithful distributive representation. 
\item $A$ has a faithful thin (i.e. squarefree) representation, equivalently, $A$ has a faithful 1-representation.
\end{enumerate}
\end{theorem}
\begin{proof}
(1)$\Rightarrow$(2), (3) follows just by using the defining representation of the incidence algebra. \\
(2)$\Rightarrow$(3) Note that ``faithful thin representation" and ``faithful 1-representation" are obviously equivalent, and equivalent to ``faithful distributive representations" by Proposition \ref{p.6.1} in the acyclic case. \\
(3)$\Rightarrow$(1) As noted above, by the characterization of distributivity, such a faithful 1-representation $M$ will be distributive. Let $\varphi:A\hookrightarrow M^n$ be an injective morphism. Then for every idempotent $e$, we have $\dim(eM)=\dim(\Hom(Ae,M))=[M:S_e]=1$ - the multiplicity in $M$ of the simple $S_e$ corresponding to $e$. Let $eM=\KK m_e$, and $\varphi(e)=(\lambda_1 m_e,\dots,\lambda_n m_e)\in M^n$, $\lambda_i\in \KK$ not all $0$. Then ${\rm ann}(e)=\bigcap\limits_i{\rm ann}(\lambda_i m_e)={\rm ann}(m_e)$, and so $Ae\cong Am_e$. But $Am_e$ is distributive as a submodule of $M$, and therefore so is $Ae$. Moreover, we see that any non-zero morphism $\theta:Ae\rightarrow M$ from some $Ae$ to $M$ is injective, since $\theta(e)=e\theta(e)\in eM$ so $\theta(e)=\lambda m_e$ for some $\lambda$. Therefore, if $\psi:Ae\rightarrow Af$ is a nonzero morphism, composing it to $Af\hookrightarrow M$ yields a non-zero - and thus injective - morphism. Hence, $\psi$ must be itself injective. This shows that $A$ is locally hereditary (and so also acyclic). The result follows from the characterization of deformations of incidence algebras Theorem \ref{t.definc} (2).
\end{proof}

The equivalence between (1) and (3) answers a question that was asked to us by Marcelo Aguiar (Maguiar@math.cornell.edu). We note that the above conditions are representation-theoretic (categorical) in nature, and hence, if we drop the condition that $A$ is basic, we obtain characterizations of structural matrix algebras, or equivalently, of incidence algebras of quasi-ordered sets.

\begin{corollary}
Let $A$ be an algebra which is Schur, i.e. $\End(S)=\KK$ for every simple $A$-module (in particular, this always holds when $\KK$ is algebraically closed). Then $A$ is a structural matrix algebra or equivalently, an incidence algebra of a quasi-ordered set if and only if either of the conditions (2) or (3) of Theorem \ref{t.inc} holds.
\end{corollary}

\subsection{Application: further characterizations}\label{s.appl}

We note how our results also recover the results of R. Feinberg \cite{F1,F2}. 

We will need the following representation theoretic lemma which may be well known, but we could not find a reference. It may be useful outside the scope of this paper, as it gives a criteria for a projective to be a Schur module. We will need the following well known remark. If $I$ is an ideal of the algebra $A$, let $B=A/I$ and consider the functors $Res:B{\rm-Mod}\rightarrow A{\rm-Mod}$ the restriction of scalars via $A\rightarrow A/I=B$, and $F:A{\rm-mod}\rightarrow B{\rm-Mod}$ be $F(M)=M/IM\cong A/I\otimes_A M$ and correspondingly defined on morphisms. Then $(F,Res)$ is an adjoint pair, and since $Res$ is exact, $F$ preserves projectives; thus, $P_0=P/IP=F(P)$ is projective over $B$.


\begin{lemma}
Let $A$ be a basic Shurian algebra, $P=P(S)$ be a projective indecomposable $A$-module with top $S$. If the center of the algebra $A/\ann(P')$ is semisimple for any quotient $P'$ of $P$, then $[P(S):S]=1$, equivalently, $\End(P)=\KK$ (i.e. $P$ is a Schur module). 
\end{lemma}
\begin{proof}
Assume, by contradiction, that $[P(S):S]\geq 2$, so $JP$ has $S$ as a subquotient. Then there exists a diamond $D$ - a module with simple top and simple socle - such as both the top and the socle (of $D$) are isomorphic to $S$ (for example, if $X\subseteq P$ is a submodule such that $S\hookrightarrow P/X$, and if $X$ is taken maximal with this property, then it is easy to see that $S=Soc(P/X)$). Let $I=\ann(D)$; there is a projection $P\rightarrow D$, which induces a surjection $P_0=P/IP\rightarrow D$. Let $B=A/I$, $J=Jac(B)$ and $B=P_0\oplus\bigoplus\limits_{i=1}^n P_i$ ($P_0$ is projective indecomposable over $B$), with $P_i=Be_i$ and $e_i$ primitive idempotents (note that $B$ is basic pointed too). \\
Let $lw(X)$ denote the Loewy length of a $B$-module $X$, and let $l=lw(D)$ (this coincides with the Loewy length as an $A$-module). 
Since $D$ is $B$-faithful, there is an embedding $\phi:B\hookrightarrow D^k$. Let $\phi(e_i)=(m_{ij})_{j=1,\dots,k}$ and let $\pi_j$ denote the canonical projections of $D^k$. For $i\geq 1$, note that $\pi_j\phi(e_i)\subseteq JD$ since otherwise, since $D$ is local, we'd get $\pi_j\phi(P_i)=\pi_j\phi(Ae_i)=D$, and so $D$ would be a quotient of $P_i$. In this case, the simple $S={top}(D)$ would be isomorphic to the top of $P_i$, which is not possible. Hence, $P_i\cong \phi(P_i)\subseteq (JD)^k$, which shows that $lw(P_i)\leq l-1$. Also, since $P_0\hookrightarrow D^k$, $lw(P_0)\leq l$, and in fact $lw(P_0)=l$ because of the projection $P_0\rightarrow D$. Moreover, the embedding $P_0\hookrightarrow D^k$ shows that the socle of $P_0$ is $S^t$ for some $t$. Since $lw(P_0)=l$, $J^lP_0\neq 0$ and $J^{l+1}P_0=0$; let $x\in J^lP_0$. Then $x\in Soc(P_0)$ since $Jx=0$, and so $\KK x\cong S$. \\
Now, let $M=JP_0\oplus\bigoplus\limits_{i=1}^kP_i$. Since $lw(JP_0),lw(P_i)< l$ (since $P_0$ is local $lw(JP_0)<lw(P_0)$), we get $lw(M)<l$ (in fact, $lw(M)=l-1$) and therefore $J^lM=0$. Thus, $xM=0$ since $x\in J^l$. At the same time, if $b\in M$, write $b=be_0+\sum\limits_{i=1}^nbe_i\in JP_0\oplus\bigoplus\limits_{i=1}^nP_i$ (component-wise decomposition of $b$) and since $e_ix=0$ for $i\geq 1$ ($\KK x\cong S\cong Top(P_0)\not\cong Top(P_i)$ for $i\geq 1$), we get $bx=be_0x$ (alternatively, use $x=e_0x$). But $be_0\in JP_0\subset J$, and since $\KK x$ is simple, $be_0x=0$; hence $bx=0$. \\
To summarize, we have $xM=Mx=0$, and also, $e_0x=x=xe_0$ (since $x\in P_0=Ae_0$ and $\KK x\cong S$). Now we see that $x\in Z(B)$: write $B=\KK e_0\oplus M$ as vector spaces; if $b\in B$, write $b=\lambda e_0+m$ ($m\in M$) so $bx=\lambda e_0x+mx=\lambda x$ and $xb=\lambda xe_0+xm=\lambda x$ (in fact, $J^n\subset Z(B)$). But this means $Z(B)$ is not semisimple, since it is commutative and contains nilpotents: $x^2=0$. This is a contradiction and the proof is finished.  
\end{proof}

We can recover the main result of \cite{F2}; the results there are formulated for the natural extension to infinite dimensional topological pseudocompact algebras - i.e. algebras which are dual to coalgebras; \cite{DNR}. We have only been interested here in the finite dimensional case.

\begin{corollary}[Feinberg]
Let $A$ be a finite dimensional algebra. Then $A$ is an incidence algebra of a quasi-ordered set if and only if the following three conditions hold:\\
(i) $A$ has a faithful distributive representation, \\
(ii) the blocks of $A/J$ are matrix algebras over $\KK$ and \\
(iii) the center of any quotient $A/I$ is a product of $\KK$'s. 
\end{corollary}
\begin{proof}
The block condition (ii) is obviously equivalent to $A$ being Schur (i.e. $\End(S)=\KK$ for every simple module $S$); up to Morita equivalence, it is enough to assume $A$ is basic, and hence pointed. By the previous Lemma, condition (iii) shows that $[P(S):S]=1$. Now, since $A$ has a faithful distributive representation $M$, by Proposition \ref{p.6.1}, we see that $M$ is a 1-representation (multiplicity free). Hence, the direct statement follows by the structure Theorem \ref{t.inc}. Conversely, the center condition is a known (and computationally easily verified) property of incidence algebras \cite{DS}.
\end{proof}

One may wonder whether a condition regarding the existence of an indecomposable distributive 1-representation could perhaps characterize incidence algebras. This is motivated by the observation that when the poset $P$ is connected (i.e. its Hasse quiver is so), then it is not hard to see that the defining representation of $I(P,\KK)$ is indecomposable. Nevertheless, note that the acyclic quiver $\xymatrix{\bullet \ar@/^/[r] \ar@/_/[r] & \bullet}$ has infinitely many indecomposable 2-dimensional representations which are 1-representations, but its quiver algebra is not an incidence algebra. It also has indecomposable faithful representations (for example, the non-simple projective indecomposable module). 

\subsection*{Classification of distributive/thin representations}

By the previous results, over an acyclic algebra, a representation is distributive exactly when it is thin. In this subsection we give the classification of thin modules (equivalently, distributive modules) over incidence algebras. Later, this will be used to give complete invariants for thin modules over arbitrary algebras, and to classify their thin representations. 

Let $A=I(P,\KK)$ be an incidence algebra. If $M$ is a distributive representation (not necessarily faithful), let $I=\ann(M)$. The algebra $A$ has only finitely many ideals, the structure of which is well understood; $I$ has a basis consisting of elements among the $f_{xy}$'s, and let $B$ be this basis of $I$, and $S$ be its complement in the set $\{f_{xy}| x\leq y\}$. The ideal condition translates to the property that $S$ is closed under subintervals: if $f_{xy}\in S$ and $[a,b]\subset [x,y]$ (meaning that if $x\leq a\leq b\leq y$, then $f_{ab}\in S$), then $f_{ab}\in S$. Indeed, if $f_{ab}\not\in S$, then $f_{ab}\in I$ so $f_{xy}=f_{xa}f_{ab}f_{by}\in I$ too. 

Note also that if $f_{xy}, f_{yz}\in S$, then $f_{xz}\in S$. To see this, let $B=A/I$; this is (isomorphic to) an incidence algebra, since it possesses the faithful distributive representation $M$, and the cosets $E=\{\overline{f_{xx}}| f_{xx}\in S\}\subset B$ form a set $E$ of primitive orthogonal idempotents (and hence, they are easily seen to be linearly independent; a similar more general argument is used, for example, in \cite{DIN}). In fact, the idempotents in $\{\overline{f_{xx}}|f_{xx}\not\in I\}$ are also primitive idempotents, since the spaces $\overline{f_{xx}}B\overline{f_{xx}}$ are at most 1-dimensional (because $\dim(f_{xx}Af_{xx})= 1$).

As noted before, in an incidence algebra $B$, for any three primitive idempotents $e,f,g$, if $0\neq s_{ef}\in eBf$ and $0\neq s_{fg}\in fBg\neq 0$, then $0\neq s_{ef}\cdot s_{fg}\in eBf\cdot fBg=eBg$ (this is easily extended for non-primitive idempotents too). Hence, if $f_{xy},f_{yz}\in S$ then $\overline{0}\neq \overline{f_{xy}}\in \overline{f_{xx}}B\overline{f_{yy}}$ and $\overline{0}\neq \overline{f_{yz}}\in \overline{f_{yy}}B\overline{f_{zz}}$, and so $\overline{f_{xz}}=\overline{f_{xy}}\cdot \overline{f_{yz}}\neq \overline{0}$, i.e. $f_{xz}\in S$. 


The last part shows that if $f_{xx},f_{yy}\in S$, we may define $x\leq_S y$ if and only if $f_{xy}\in S$, and obtain that $(S,\leq_S)$ is a subposet of $(P,\leq)$ (note that if $f_{xx},f_{yy}\in S$ it does not necessarily follow that $f_{xy}\in S$; for example, one may simply take $I$ to be the Jacobson radical, which will make all relations in $P$ disappear in $S$, i.e. all vertices will become incomparable). 

\begin{definition}
The poset $(S,\leq_S)$ uniquely associated to $M$ as above will be called the support of $M$, and we will write $S={\rm Supp}(M)$.
\end{definition}

Moreover, the comments above show that $(S,\leq_S)$ is a subposet of $(P,\leq)$, which has the property that it is closed under subintervals; that is, if $x,y\in S$ and $x\leq_S y$, then for any $z\in P$ with $x\leq z\leq y$, we have $z\in S$ and $x\leq_S z\leq_S y$. This motivates the introduction of the following

\begin{definition}\label{d.closed}
A subposet $(S,\leq_S)$ of $(P,\leq)$ (meaning simply that $S\subseteq P$  and if $x\leq_S y$ then $x\leq y$) is said to be a closed subposet if $S$ is closed under subintervals in the above sense.
\end{definition}

Combinatorially a closed subposet is obviously realized in the following way. Consider a set $X=\{[x_i,y_i]|i=1,\dots,t\}$ of intervals of $P$, so $[x_i,y_i]=\{z\in P| x_i\leq z\leq y_i\}$. We perform two changes to this set $X$. First, keep only the intervals that are maximal, i.e. if $[x_i,y_i]\subseteq [x_j,y_j]$ then eliminate $[x_i,y_i]$. Furthermore, eliminate ``overlaps": if $x_i\leq x_j\leq y_i\leq y_j$, then replace the pair $[x_i,y_i], [x_j,y_j]$ by $[x_i,y_j]$. Continue until none of the two changes can be done anymore. Let $S$ be the set of all $z$ in one of the intervals in $X$, with relations given by the relations in these intervals; that is $a\leq_S b$ if $x,y\in [x_i,y_i]$ for some $i$. This is obviously a subposet of $P$ which is closed under subintervals. Conversely, it is easy to see that every such closed subposet arises this way: if $S$ is a closed subposet, consider the set of maximal intervals of $S$. The last ``overlap" e;elimination procedure above is necessary as $S$ must be transitive, so $x\leq_S y$ and $y\leq_S z$ implies that $x$ and $z$ {\it are related} in $S$. 

We saw that a distributive representation $M$ of $A=I(P,\KK)$ becomes a faithful distributive representation of $I((S,\leq_S),\KK)$, and thus is completely determined by some 1-cocycle $\alpha\in H^1(\Delta(S,\leq_S),\KK^*)$. Hence, we have

\begin{theorem}\label{t.discls}
The distributive representations of an incidence algebra $A=I((P,\leq),\KK)$, equivalently, the thin representations, are in 1-1 correspondence with the set
$$\bigsqcup\limits_{(S,\leq_S) {\rm\,closed}}H^1(\Delta(S,\leq_S),\KK^*)$$
where the disjoint union is realized over the closed sub-posets of $(P,\leq)$. 
\end{theorem}


\section{Thin representation theory}\label{s.thin}

\subsection*{Generic classification of distributive/thin representations}

We note the following corollary, which states that, in the acyclic case, {\it every} possible distributive representation comes precisely as one realized via the defining representation of an incidence algebra, and so it is combinatorial in nature; also, any thin representation of any finite dimensional algebra can be realized as such a defining representation of an incidence algebra. Hence, our results show that this construction is generic and covers the most general case. More precisely, this is formulated as follows.

\begin{theorem}\label{t.Ann}
Let $A$ be a finite dimensional basic pointed algebra, and assume that either $A$ is acyclic and $M$ is distributive, or that $M$ is a thin module. Then $A/\ann(M)$ is an incidence algebra of some poset $P$; moreover, there are suitable bases $(f_{xy})_{x\leq y, \, x,y\in P}$ of $A/\ann(M)$ and $(m_{x})_{x\in P}$ of $M$, with respect to which $M$ is the defining representation of $A/\ann(M)=I(P,\KK)$, that is, $f_{xy}\cdot m_y=m_x$. Equivalently, after composing to a suitable algebra automorphism $\varphi:A/\ann(M)\rightarrow A/\ann(M)$, the representation ${}_{\varphi}M$ with the induced structure via $\varphi$ can be seen as the defining representation of $A/\ann(M)\cong I(P,\KK)$. 
\end{theorem}
\begin{proof}
Obviously, $M$ is a faithful representation of the algebra $A/\ann(M)$; the hypothesis states that either $A/\ann(M)$ is acyclic and $M$ is distributive, or that $M$ is thin; in both cases, Theorem \ref{t.inc} readily implies that $A/\ann(M)\cong I(P,\KK)$ for some poset $P$. For the last part, fix an ``incidence algebra basis" $(f_{xy})_{x\leq y}$ of $A/\ann(M)$ (so $f_{xy}f_{zt}=\delta_{yz}f_{xt}$) and let $M$ be given by the 1-cocycle $\alpha$; that is, $\alpha$ is ``multiplicative": $\alpha_{xz}=\alpha_{xy}\alpha_{yz}$, $x\leq y\leq z$ and $M$ has a basis $m_x$ with $f_{xy}\cdot m_y=\alpha_{xy}m_x$. Then, letting $e_{xy}=\frac{1}{\alpha_{xy}}f_{xy}$, we see that $e_{xy}\cdot e_{yz}=e_{xz}$ and $e_{xy}\cdot m_y=m_x$. Thus, after this change of basis in $A/\ann(M)$, $M$ is viewed (realized) as the defining representation of $I(P,\KK)$ (with poset basis $e_{xy}$, $x\leq y$).
\end{proof}

The previous theorem implies that distributive representations of acyclic algebras, as well as thin representations of arbitrary algebras $A$, can be classified in terms of their annihilator $I$ and the first cohomology of the poset associated to the Ext quiver of $A/I$, and yields a complete invariant for such modules. The next corollary then follows immediately from these considerations.

\begin{corollary}
Any distributive representation $M$ of an acyclic algebra $A$, and any thin representation $M$ of an arbitrary finite dimensional algebra $A$ (which is basic pointed), is completely determined up to isomorphism by $I=\ann(M)$ and an element $\alpha\in H^1(\Delta(P_{A/I}),\KK^*)$, where $P_{A/I}$ is the poset associated canonically with the acyclic quiver of the algebra $A/I$ (and consequently, $A/I\cong I(P_{A/I})$). Hence, if ${\mathcal D}$ be the set of ideals of $A$ which are annihilators of some thin representation (equivalently, a distributive representation when $A$ is acyclic), then the set of distributive (or thin) representations is in 1-1 correspondence with (and is classified by) the set
$$\bigsqcup\limits_{I\in {\mathcal D}}H^1\left(\Delta(P_{A/I}),\KK^*\right)$$
\end{corollary}





\vspace{.5cm}

\subsection{Classification of thin modules of finite dimensional algebras}

As we have seen before, to classify thin representations it is useful to also have a way to classifying such ideals $I$ which are annihilators of some thin module. For the case of incidence algebras, this was done above, and such ideals are parametrized by subposets closed under subintervals. Using the methods and ideas before, we extend that now to general finite dimensional algebras. It will be convenient to first deal with quivers, and then move to quivers with relations.

Let $Q$ be an arbitrary quiver (not assumed acyclic, nor finite). We recall that a quiver representation $V$ is said to be {\it nilpotent} if it is annihilated by cofinite monomial ideal of the path algebra $\KK[Q]$ of $Q$; it is locally nilpotent if every $v\in V$ is annihilated by such an ideal; see \cite{CKQ,CM} (see also \cite{DIN,DIN2} for the relation with comodules over the quiver coalgebra). Hence, a finite dimensional such $V$ is a representation of some finite dimenisonal quotient $A$ of $\KK[Q]$; conversely, any finite dimensional representation of a finite dimensional algebra is a nilpotent representation of a quiver $Q$. Obviously, to understand finite dimensional thin representations of finite dimensional algebras, it will be enough to reduce to this case. 

If $V$ is a finite dimensional nilpotent representation of $Q$, then $V$ is thin precisely when $\dim(V_a)\leq 1$ for all $a\in Q_0$; this is because when $V$ is nilpotent, the dimension vector $\underline{\dim}(V)$ captures the multiplicity of simples in $V$. Let $\Gamma=\Gamma_V$ be the subquiver of $Q$ consisting of vertices for which $V_a\neq 0$, and arrows $x:a\rightarrow b\in \Gamma_1$ such that $\lambda_x:V_a\rightarrow V_b$ is non-zero ( so $\lambda_x\in\KK^\times$). 

\begin{definition}
Given a thin representation $V$ of $Q$, the quiver $\Gamma=\Gamma_V$ will be called the support of $V$.
\end{definition}
 
Note that $\Gamma$ is acyclic. This follows either by the results of the previous section, since $\Gamma$ is precisely the Ext quiver of the algebra $\KK[Q]/{\rm ann}(V)$, which is an incidence algebra, or directly: since any path in $\Gamma$ acts as a non-zero element on some vertex, the nilpotent condition means that there is a bound on the length of paths in $\Gamma$, equivalently, $\Gamma$ is acyclic. 

As before, consider $\Gamma$ as a 1-dimensional simplicial set with vertices $\Gamma_0$ and 1-simplices the arrows $\Gamma_1$; the cochain complex $0\rightarrow \Hom_\ZZ(\ZZ\Gamma_0,\KK^\times)\rightarrow \Hom_\ZZ(\ZZ\Gamma_1,\KK^\times)\rightarrow 0$ realizes cohomology with coefficients in $\KK^\times$. A thin representation $V=(\lambda_x)_{x\in \Gamma_1}$ can be regarded as simply a 1-cocycle in on $\Gamma$ with coefficients in $\KK^\times$; as before in the previous sections, two representations $V=(\lambda_x)_x$ and $W=(\mu_x)_x$ are isomorphic precisely when there are $\alpha_a\in \KK^\times$ such that $\alpha_a\lambda_x=\mu_x\lambda_b$ for all $x:a\rightarrow b \in \Gamma_1$. Let $V=V(\Gamma,[\lambda])$ be the corresponding representation of support $\Gamma$ and given by $[\lambda]\in H^1(\Gamma,\KK^\times)$. We obtain the following analogue of \ref{t.discls}.

\begin{theorem}\label{t.thin}
The nilpotent finite dimensional thin representations of a quiver $Q$ are parametrized bijectively by pairs $(\Gamma,[\lambda])$, where $\Gamma$ is an acyclic subquiver of $Q$, and $[\lambda]\in H^1(\Gamma,\KK^\times)$.  
\end{theorem}

Moreover, a precise algorithm can be obtained for parametrizing representations. For any acyclic quier $\Gamma$, let $T$ be a spanning tree of $\Gamma$ (a subquiver whose graph is a tree and contains all vertices of $\Gamma$). Note that $H^1(\Gamma,\KK^\times)=(\KK^\times)^{e(\Gamma)-v(\Gamma)+1}$, where $e(\Gamma),v(\Gamma)$ are the number of edges and number of vertices, respectively (this is well known and follows, for example, from the above cochain complex computing $H^1$ using the spanning tree, or by computing the Euler-Poincare characteristic, etc.). An element in $H^1(\Gamma,\KK^\times)$ is thus completely determined by placing $1$ at the arrows of $T$, and arbitrary $\lambda\in\KK^\times$ at the other arrows in $\Gamma$.

\subsection*{Finite dimensional algebras.} Now, in order to generalize to finite dimensional algebras, let $A=\KK[Q]/I$ be a finite dimensional algebra given by a finite quiver $Q$ with relation ideal $I$. Any thin representation $V$ of $A$ will be a thin nilpotent representation of $Q$ of support $\Gamma\subset Q$; hence, we are looking at such thin representations $V(\Gamma,[\lambda])$ of $Q$ which are annihilated by $I$ (so they become $A$-modules), for various possible supports $\Gamma$. Fix a support $\Gamma$ and let $T$ be a fixed spanning tree of $\Gamma$. To each arrow $x$ of $\Gamma-T$ assign an indeterminate invertible variable $v_x$ and consider the algebra of Laurent polynomials $B=\KK[(x_v^{\pm 1})_{v\in \Gamma-T}]$ in $|\Gamma-T|$ variables. Let $q_s=\sum\limits_i \beta_{si} p_{si}$, $s=1,2,\dots,t$ be generators of the ideal $I$ (any cofinite ideal in $\KK[Q]$ can be generated with finitely many elements), and consider the polynomials $R_s\in B$ obtained from $q_s$ by replacing the arrows $x$ on the paths $p_{si}$ either by $1$ if $x$ is an arrow in $T$, or by its corresponding variable $v_x$ from $B$ if $x$ is not in $T$.  The equations $R_s=0$ are equations of an affine algebraic variety which gives a one-to-one parametrization of thin $A$-modules of support $\Gamma$, whose coordinate algebra is $\KK[(x_v^{\pm 1})_{v\in\Gamma-T}]/(R_1,\dots,R_t)$. Putting all these varieties together gives a variety parametrizing all thin $A$-modules. Furthermore, the algebra $\KK[x_v|v\in\Gamma-T]/(R_1,\dots,R_t)$ also gives a one-to-one parametrization 
of the thin modules whose support {\it contains} $\Gamma$, since in this case one may allow some of the variables $x_v$ for $v\in (\Gamma-T)_1$ to become $0$. For this, one needs only note that doing such a replacement of an $x_v$ is equivalent to simply ``deleting" that arrow from the quiver and all the paths in which it appears from the relation ideal; in this case, $T$ remains a spanning tree of the resulting quiver. We follow this on an example. 

\begin{example}
Consider the quiver 
$$Q:\,\,
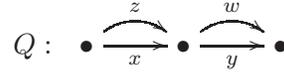
\xymatrix{\bullet \ar_{x}[r]\ar@/^2ex/^z[r] & \bullet \ar_{y}[r]\ar@/^2ex/^w[r] & \bullet  }$$ 
and let $A$ be the finite dimensional algebra obtained from this quiver with relation $\alpha xy+\beta xw+\gamma zy+\delta zw=0$. For support $\Gamma=Q$, a spanning tree consists of the path $zw$; substituting $1$ for $z$ and $w$ and arbitrary $x,y\in \KK^\times$ yields the variety $\alpha xy+\beta x+\gamma y+\delta=0; x,y\neq 0$ giving a one-to-one parametrization  of thin $A$-modules of support $Q$. Thin modules of support $Q-\{x\}$, respectively, $Q\setminus\{y\}$ correspond to setting $x=0$, and respectively, $y=0$, and are bijectively parametrized by equations $\gamma y+\delta=0; y\neq 0$ and $\beta x+\delta=0; x\neq 0$ respectively. Similarly, the equation $\delta=0$ gives the thin modules of support $Q-\{x,y\}$ (either no such module exists if $\delta\neq 0$ or otherwise there is exactly one). Thus the variety $\alpha xy+\beta x+\gamma y+\delta=0$ gives a bijective parametrization for the thin $A$-modules of support $\Gamma$ containing the path $zw$. Several other families of thin modules exist in this case (with other possible supports). Note that the affine variety $\alpha xy+\beta xw+\gamma zy+\delta zw=0$ gives a space parametrizing all thin modules (though not bijectively). The torus $\KK^\times\times \KK^\times\times \KK^\times$ acts on this variety as an usual action on a representation variety $(a,b,c)\cdot (x,y,z,w)=(b^{-1}xa,c^{-1}yb,b^{-1}za,c^{-1}wb)$, making this a toric variety whose orbits parametrize thin modules of this algebra $A$. 
\end{example}

Extending the considerations of the previous example, consider the ``double $A_n$" quiver 

$$AA_n:\,\,
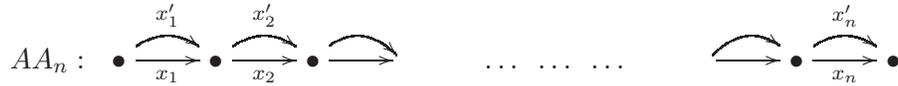
\xymatrix{\bullet \ar_{x_1}[r]\ar@/^2ex/^{x_1'}[r] & \bullet \ar_{x_2}[r]\ar@/^2ex/^{x_2'}[r] & \bullet \ar[r]\ar@/^2ex/[r] & & \dots \,\,\,\dots\,\,\, \dots & \ar[r]\ar@/^2ex/[r] & \bullet \ar_{x_n}[r]\ar@/^2ex/^{x_n'}[r] & \bullet}$$

We will have the following observation. Let $p(x_1,\dots,x_n)=\sum\limits_{1\leq i_1<\dots <i_t\leq n}\alpha_{i_1,\dots,i_t}x_{i_1}\dots x_{i_t}$ be a multilinear polynomial. The thin representations of the algebra $\KK[AA_n]/\langle p\rangle$	- that is, of the quiver $A_n$ with relations given by $p$ - and whose support contain the (quiver generated by the) path $x_1'x_2'\dots x_n'$ are parametrized by the variety $\KK[x_1,\dots,x_n]/p(x_1,\dots,x_n)$. Obviously, this idea can be extended to any multilinear variety, and we have the following theorem; its proof is a direct extension of the previous ideas, and is left to the reader. It shows that {\it any multilinear variety can be regarded as a moduli space of thin representations}.

\begin{theorem}\label{t.mthin}
Let $p_1,\dots,p_s$ be multilinear polynomials in $n$ variables $x_1,\dots,x_n$, and consider the affine linear variety $V:p_1=0;\dots;p_n=0$ and its coordinate algebra $H=\KK[x_1,\dots,x_n]/(p_1,\dots,p_s)$. Then points of the affine variety $V$ represent (in the sense they give a 1-1 parametrization) for the thin modules of the algebra $A_{p_1,\dots,p_s}=\KK[AA_n]/(p_1,\dots,p_s)$ (the quiver $AA_n$ with relations $p_1,\dots,p_n$) and whose support contains $x_1'x_2'\dots x_n'$.
\end{theorem}

\vspace{.5cm}

\section{Graphs and incidence matrices}\label{s.graphs}


We note here the connections to graphs and their incidence matrices, as well as another immediate application to computing the orbits of the action by conjugation of the diagonal subgroup $(\KK^\times)^n=Diag_n(\KK)$ of $Gl_n(\KK)$ on $M_n(\KK)$, where $\KK$ is a field which is not necessarily algebraically closed. These are applications of the main methods used above.

Let $A=(a_{ij})_{i,j}$ be an $n\times n$ matrix. The ``zero-nonzero pattern" of $A$, that is, the set of positions $(i,j)$ at which $a_{i,j}\neq 0$, can be readily encoded by an oriented graph $\Gamma_A$ with $n$ vertices numbered $1,\dots,n$ and an arrow $i\rightarrow j$ whenever $a_{i,j}\neq 0$. We call this the incidence graph of $A$; when $A$ is a $0,1$-matrix, then $A$ is just the incidence matrix of the graph $\Gamma_A$). This pattern is, of course, a common natural object to look at combinatorially, as it is for incidence algebras as well \cite{MW}. It is obviously invariant under conjugation by invertible diagonal matrices. Moreover, the matrix $A$ can be interpreted as a special type of representation $V$ of the quiver $\Gamma_A$, where at each vertex a 1-dimensional space is placed, and $a_{ij}\in\KK^\times$ are the morphisms corresponding to $i\rightarrow j$. This representation, however, does not have to be (locally) nilpotent (its annihilator does not have to contain a power of the arrow ideal; see above and \cite{DIN}), in which case the vertices of the quiver will not necessarily represent simple modules in the composition series of $V$, as the following example shows.

\begin{example}
Consider the quiver
$$\xymatrix{a\ar@/^/[r]^{x} & b\ar@/^/[l]^y}$$
and the representation $V$ as above with 1-dimensional spaces $V_a=V_b=\KK$ and with $x$ acting as the scalar $\mu\neq 0$ and $y$ as the scalar $\lambda\neq 0$. It is not difficult to see that, in fact, this is a simple module over this quiver; modulo the annihilator of $V$, the algebra is just isomorphic to $M_2(\KK)$. 
\end{example}

The case when the representation is nilpotent is closely related to our general theory. As noted before, $V$ is nilpotent exactly when the quiver $\Gamma_A$ is acyclic, in which case $V$ is thin and is completely determined by first cohomology; hence, it is not surprizing that the orbits of the above mentioned action are also described by cohomology. Given $A$, consider the relation on $P=\{1,\dots,n\}$ given by $i\preceq j$ if there is a (possibly trivial) path in $\Gamma_A$ from $i$ to $j$. This is reflexive and transitive, but not necessarily antisymmetric; hence, it is a quasi-order. It is antisymmetric exactly when $\Gamma_A$ is acyclic; in that case, by the results of the previous section, we have that $\KK[\Gamma_A]/{\rm ann}(V)$ is isomorphic to the incidence algebra of $(P,\preceq)$. This, however, holds in full generality, and gives an interesting connection between incidence matrices, incidence algebras, and thin representations. We prove first a more general fact. 

\begin{lemma}
Let $Q$ be a quiver, and $V=\left((V_a)_{a\in Q_0};(f_x)_{x\in Q_1}\right)$ a finite dimensional representation of $Q$ which has $V_a=\KK$ for every vertex $a\in Q$. Then $V$ is a thin representation. Moreover, $\KK[Q]/{\rm ann} V$ is isomorphic to the incidence algebra of the preorder $\preceq$ on $Q$, defined by $a\preceq b$ if there is a path $p=x_1x_2\dots x_n$ from $a$ to $b$ such that the corresponding maps $f_{x_i}$ defining the representation $V$ are non-zero (scalars). 
\end{lemma}
\begin{proof}
Let $V=\bigoplus\limits_{a\in Q_0}V_a$ and fix a basis $(v_a)_a$ of $V$ such that $V_a=\KK v_a$. Consider the representation map $\rho:\KK[Q]\rightarrow \End_\KK(V)=M_n(\KK)$, with identification via a basis $\{v_a|a\in Q_0\}$. If $a\preceq b$, then there is a path $p$ from $a$ to $b$ in $Q$, and it acts by a nonzero scalar taking $v_a$ to $\lambda_{ab}v_b$ ($\lambda_{ab}=f_{x_1}f_{x_2}\dots f_{x_n}$ if $p=x_1x_2\dots x_n$). This path $p$ acts as zero on any other $v_b$, $b\neq a$. It means that $\KK E_{ab}\subseteq {\rm Im}(\rho)$, where $E_{ab}$ are the usual matrix units in $M_n(\KK)$. On the other hand, we also quite obviously have ${\rm Im}(\rho)\subseteq\bigoplus\limits_{a\preceq b}\KK E_{ab}$, and hence ${\rm Im}(\rho)$ is a structural matrix algebra. The assertion now follows from the results of the previous section.
\end{proof}

Note that for such a quiver representation, although thin, the simple modules in its composition series are not necessarily 1-dimensional. On the other hand, not all thin modules of an arbitrary quiver will be of this type, as the previous example shows; in fact, classifying all thin modules of a quiver would include the classification of finite dimensional simple modules of arbitrary quivers and free algebras. (The classification of the previous subsection suffices though to understand thin modules over any finite dimensional algebra).  

The following corollary is the announced connection between graphs, incidence matrices, and thin representations. 

\begin{corollary}\label{c.gmt}
Let $A$ be an $n\times n$ matrix, let $\Gamma_A$ be the graph with $n$ vertices associated to $A$ via the zero-nonzero pattern of $A$, and let $V$ be the representation of $\Gamma_A$ associated to this matrix by placing $\KK$ at every vertex and $a_{ij}\neq 0$ for every arrow $i\rightarrow j$. Then $V$ is a thin $\Gamma_A$-representation, and $\KK[\Gamma_A]/{\rm ann}(V)$ is isomorphic to the incidence algebra associated to the pre-ordering $i\preceq j$ on $\{1,\dots,n\}$ defined by $i\preceq j$ if there is a path in $\Gamma_A$ from $i$ to $j$.  
\end{corollary}

From the proof above shows yet another interesting relation between incidence matrices of graphs and incidence algebras. If $M$ is the incidence matrix of the graph $\Gamma$, which is taken to be a Schurian quiver (i.e. for every pair of vertices $a,b$, there is at most one arrow with source $a$ and target $b$), then $M^n$ is a matrix with non-negative entries. The set of pairs $(i,j)$ for which in some power of $M$ the $(i,j)$ entry is non-zero is precisely given by the quasi-order, i.e. these are exactly the pairs $(i,j)$ for which $(i\preceq j)$. The structural subalgebra of $M_n(\KK)$ generated by this quasi-order is precisely the smallest structural subalgebra containing all the matrices $E_{ij}$ with $(i,j)$ such that $i\rightarrow j$ is an arrow in $\Gamma$; this corresponds to the ``quasi-order" generated by the oriented graph $\Gamma$.

\subsection*{The action of the diagonal group.}

To determine the orbits of the action by conjugation of the diagonal subgroup $Diag_n=(\KK^\times)^n$ of $Gl_n(\KK)$ on $M_n(\KK)$, one can look again at cohomology. As before, consider the 1-dimensional simplicial set $T$ whose 0 and 1-simplices are given by vertices and arrows of $\Gamma_A$, the associated simplicial chain complex
$$0\rightarrow \ZZ T_1\rightarrow \ZZ T_0\rightarrow 0$$
and the cochain complex  $0\rightarrow \Hom_{\ZZ}(\ZZ T_0,\KK^\times )\rightarrow \Hom_\KK(\ZZ T_1,\KK^\times)\rightarrow 0$ giving cohomology with coefficients in $\KK^\times$. A matrix $A=(a_{ij})$ can be interpreted as an element $A\in \Hom_\KK(\ZZ T_1,\KK^\times)$, and the computation before for quiver representations shows that if $A,B$ are two such matrices and $D={\rm Diag}(\lambda_1,\dots,\lambda_n)\in Diag_n$, then $DA=BD$ means that the representations of $\Gamma_A=\Gamma_B$ associated to $A$ and $B$ are isomorphic via $D$, so $a_{ij}^{-1}b_{ij}=\lambda_i\lambda_{j}^{-1}$, equivalently, $A$ and $B$ are equal in $H^1(T,\KK^\times)$. For simplicity, we will just denote this group by $H^1(\Gamma_A,\KK^\times)$. As noted before (by universal coefficients, for example), this group is a torus: $H^1(\Gamma_A,\KK^\times)=(\KK^\times)^t$, where $t=e(\Gamma_A)-v(\Gamma_A)+1=e(\Gamma_A)-n+1$, and $e(\Gamma_A),v(\Gamma_A)$ are the number of vertices and edges of $\Gamma_A$, respectively. We thus obtain. 

\begin{theorem}\label{t.lin}
Up to conjugation by the diagonal group $Diag_n=(\KK^\times)^n$, any matrix $A\in M_n(\KK)$ is completely determined by its zero-nonzero pattern encoded by the graph $\Gamma=\Gamma_A$, and an element $\lambda\in H^1(\Gamma_A,\KK^\times)=(\KK^\times)^{e(\Gamma_A)-n+1}$. The orbits of the action by conjugation of the diagonal group $Diag_n=\KK^\times$ on $M_n(\KK)$ are thus parametrized by pairs $(\Gamma,\lambda)$, where $\Gamma$ is an oriented graph on $n$ vertices (i.e. a Schurian quiver; loops allowed but single arrows between vetrices), and $\lambda\in H^1(\Gamma,\KK^\times)=(\KK^\times)^{e(\Gamma)-v(\Gamma)+1}=(\KK^\times)^{e(\Gamma)-n+1}$. 
\end{theorem}

One can thus obtain a canonical form for conjugation by diagonal matrices, by choosing some consistent rule for picking representatives in such orbits. We display such a possible canonical form. First, one determines the graph $\Gamma_A$. Consider the lexicographic ordering on $\{1,\dots,n\}\times\{1,\dots,n\}$. This induces a total ordering on the set of arrows of the (oriented) graph $\Gamma=\Gamma_A$. Now, proceed to find a spanning tree $\mathcal{T}$ of $\Gamma$, by eliminating arrows in (say ascending) lexicographic order. Algorithmically, this means the following: start by setting $\mathcal{T}=\Gamma$ (of course, this is usually not a tree at first); find the arrow $a_1$ which is least in lexicographic order and which belongs to some (possibly not oriented) cycle in $\mathcal{T}$, and remove this from $\mathcal{T}$. Then proceed to find the next arrow $a_2$ which lays on some cycle of the current $\mathcal{T}$, and remove it from $\mathcal{T}$. Continue this until no more cycles exist in $\mathcal{T}$, at which point $\mathcal{T}$ is such a spanning tree of $\Gamma$. Then $A$ is conjugate to a unique matrix $C=C(A)=(c_{ij})_{i,j}$ which has the same zero-nonzero pattern $\Gamma$, and has $1$ at all positions $(i,j)$ corresponding to arrows $i\rightarrow j$ which are in $\mathcal{T}$; the elements of the other positions are uniquely determined by the equation $DA=CD$ for some diagonal invertible matrix $D=Diag(d_1,\dots,d_n)$, which means $c_{ij}=d_ia_{ij}d_j^{-1}$. The condition $c_{ij}=1$ for $(i\rightarrow j)$ arrow in $\mathcal{T}$ means $d_j=a_{ij}d_i$ for such pairs $(i,j)$. We may choose $d_1=1$, and successively determine all the other $d_j$'s following the path in $\mathcal{T}$ from vertex $1$ to vertex $j$, after which the $c_{ij}$ for arbitrary $i,j$ are determined. Other variants are certainly possible, depending on what is needed; for example, one may choose to remove all loops first, and then proceed with as above with the remaining arrows (although the loops will always get removed in the process). We note also that this procedure can also be done on arbitrary quivers (not only Schurian quivers) in order to algorithmically parametrize thin modules.

For an example consider the matrix $A$ below, where each of the entries $a_{ij}$ are some non-zero elements in $\KK$; its canonical form $C$  according to the algorithm above, as well as the diagonal conjugation matrix $D$, are computed. 

{\small
$$
A=\left(\begin{array}{cccc} a_{11} & a_{12} & a_{13} & 0 \\ 0 & a_{22} & 0 & a_{24} \\ a_{31} & a_{32} & a_{33} & a_{34} \\ 0 & a_{42} & 0 & a_{44} \end{array}\right)
\,\,\,\,\,
\,\,\,\,\,
D=\left(\begin{array}{cccc} 1 & 0 & 0 & 0 \\ 0 & a_{31}^{-1}a_{34}a_{42} & 0 & 0 \\ 0 & 0 & a_{31}^{-1} & 0 \\ 0 & 0 & 0 & a_{31}^{-1}a_{34}  \end{array}\right)
$$
}
$$
C=\left(\begin{array}{cccc} a_{11} & a_{12}a_{42}^{-1}a_{34}^{-1}a_{31} & a_{13}a_{31} & 0 \\ 0 & a_{22} & 0 & a_{24}a_{42} \\ 1 & a_{32}a_{42}^{-1}a_{34}^{-1} & a_{33} & 1 \\ 0 & 1 & 0 & a_{44} \end{array}\right)
$$

Below, the graph $\Gamma_A$ as well as the spanning tree $\mathcal{T}$ obtained as in the algorithm described are also depicted. In the graph $\Gamma_A$, the bold arrows are the arrows remaining after removing the other ``regular" arrows to obtain $\mathcal{T}$.

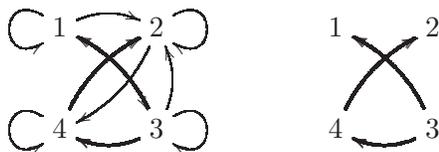
$$
\xymatrix{
1 \ar@(ul,dl)[]\ar@/^/[r]\ar@/^/[dr] & 2\ar@(ur,dr)[]\ar@/^/[dl] & & 1 & 2 \\
4 \ar@(ul,dl)[]\ar@/^/@<-.2mm>[ur]\ar@/^/[ur] & 3\ar@(ur,dr)[]\ar@/_/[u]\ar@/_/@<-.2mm>[ul]\ar@/_/[ul]\ar@/^/@<-.2mm>[l]\ar@/^/[l] & & 4\ar@/^/@<-.2mm>[ur]\ar@/^/[ur] & 3\ar@/_/@<-.2mm>[ul]\ar@/_/[ul]\ar@/^/@<-.2mm>[l]\ar@/^/[l]
}
$$

In fact, one observes from the above example that the entry $c_{ij}$ is obtained by multiplying $a_{ij}$ by a product of $a_{kl}$'s or their inverses as follows: let $\xymatrix{j=i_0 \ar@{-}[r] & i_1\ar@{-}[r] & ... \ar@{-}[r]& i_t=i}$ be the unique path in $\mathcal{T}$ from $j=i_0$ to $i=i_t$ (this is not necessarily an oriented path). Walk this path from $i_0$ to $i_t$, and for each arrow $k\rightarrow l$ traversed, multiply by either $a_{kl}$ or $a_{kl}^{-1}$ according to whether this arrow is traversed in forward or backward direction when walking from $j=i_0$ to $i=i_{t}$. For instance, in the example above, $c_{12}=a_{12}\cdot a_{42}^{-1}a_{34}^{-1}a_{31}$.




\section{An application to the accessibility conjecture}\label{s.ac}

We note here an application that yields a particular case of an interesting conjecture (or rather, open question) due to Bongartz and Ringel, that every finite dimensional module over a pointed algebra is accessible. In this generality, K. Bongartz states \cite{Bo} that the answer in general is probably negative. We answer this in the positive here for the particular case of distributive modules over acyclic algebras (and so also for thin modules), the ``combinatorial case".  

We use our setup and results of the previous section to transfer the problem from indecomposable representation to a fairly simple graph theoretical statement. First, the following is quite easy and likely known.

\begin{lemma}
Let $P$ be a poset, $A=I(P,\KK)$ and $M$ be the defining representation. Then $M$ is indecomposable if and only if $M$ is connected.
\end{lemma} 
\begin{proof}
The only if is obvious, so assume $P$ is connected. A simple computation shows that any endomorphism $\varphi\in\End_A(M)$ must have $\varphi(m_x)=\lambda_xm_x$ (use $f_{xy}\varphi(m_z)=\delta_{y,z}\varphi(m_x)$ for $x=y$); furthermore, applying $\varphi$ to $f_{xy}\cdot m_y=m_x)$ yields $\lambda_x=\lambda_y$ for $x\leq y$ for 
\end{proof}

Note that the above Lemma states that the endomorphism ring of the defining representation is nothing but $H^0(\Delta(P),
KK)$ (with ring structure as a product of $\KK$'s). Now, the reduction idea is as follows. If $M$ is a distributive representation over some acyclic algebra $A$ which is Schur (it is enough that $M$ is acyclic, i.e. $A/ann(M)$ is acyclic, equivalently, $M$ is thin), then by Theorem \ref{t.Ann}, the algebra $B=A/ann(M)$ is isomorphic to an incidence algebra of a poset $P$, and moreover there are bases $f_{xy}$ and $m_x$ ($x\in P$) of $B$ and $M$ with $f_{xy}\cdot m_z=\delta_{y,z}m_x$. Thus, the lattice of submodules of $M$ is completely determined over $B$, and hence we can reduce the problem to the case when $M$ is the defining representation of an incidence algebra $B=I(P,\KK)$. For each closed subposet $S$ of $P$ denote by $M(S)$ the defining representation of $S$; if $M$ is regarded as a poset representation, $M(S)$ is is obtained from $M$ by replacing the $\KK$ at vertices $x\notin S$ with $0$'s. In the terminology of the previous section, $M(S)$ has support $S$. Note that if $x$ is either a minimal or a maximal vertex in $P$, then $M(P\setminus \{x\})$ is a quotient, respectively, a submodule, of $M$, and these are precisely all the quotients and submodules of $M$. It is then enough to show that we can remove a min or max vertex of $P$ and keep the resulting poset $P\setminus \{x\}$ connected. Indeed, we have the following easy combinatorial statement, which is likely known but we could not locate a reference, and include a brief proof.

\begin{lemma}
Let $P$ be a connected poset. Then there is either a min or a max vertex such that $P\setminus \{x\}$ is connected.  
\end{lemma}
\begin{proof}
Let ${\rm Min}={\min(P)}$ and ${\rm Max}={\max(P)}$ be the set of min, respectively, max, vertices of $P$. Consider the bipartite (non-oriented) graph $G$ whose vertices are ${\rm Min}\cup {\rm Max}$ and an edge between $x\in {\rm Min}$ and $y\in {\rm Max}$ exists exactly when $x\leq y$. Let $T$ be a spanning tree (i.e. maximal subtree) of $G$, and let $x$ be a vertex of $T$ of degree $1$ (such $x$ exists). Assume for example $x\in {\rm Min}$ (the other case is similar), and that $x--y$ is the unique edge adjacent to $x$ in $G$, so $y\in {\rm Max}$ and $x<y$ (if $x=y$, $P$ would be disconnected). Then $P\setminus \{x\}$ is connected. Note that as $G\setminus \{x\}$ is connected (since $T\setminus\{x\}$ is so), then by the definition of $G$ it is easy to see that the vertices of $G\setminus \{x\}$ lay in the same connected component $Q$ of $P\setminus \{x\}$, since connections between these vertices are made without going through $x$. Since every $z\in P$ is connected to some $t\in{\rm Max}$ (with a path not involving $x$), the conclusion follows.   
\end{proof}

We are ready to give the result on (combinatorial part of) the general no gap/accessibility conjecture, which follows from the considerations above.  

\begin{theorem}\label{t.accessible}
Let $M$ be an indecomposable module over an acyclic algebra, which is distributive, or a thin module over some finite dimensional algebra. Then $M$ is accessible. 
\end{theorem}

\subsection*{The thin covers of a representation}

We briefly note a possible strategy to approach the general case of the accessibility conjecture. The methods of the previous sections suggest the following general procedure. Let $V$ be a finite dimensional module over a finite dimensional basic pointed algebra. For each composition series $0=V_0<V_1<\dots V_n=V$ ($A$-invariant flag), pick a basis $(v_i)_i$ compatible with this (so $V_i=V_{i-1}+\KK v_i$) and consider the representation map $\eta:A\rightarrow M_n(\KK)$, where the matrix coefficients are with respect to this basis. Then $\eta(a)=\eta_{ij}(a)$ is upper triangular for all $a$; let $B$ be the structural matrix subalgebra of $M_n$ which is the smallest of all structural subalgebras containing $A$. In other words, we introduce a relation such that $i\leq j$ whenever $\eta_{ij}\neq 0$, and we transitively complete this to an order relation. Note that $B$ will be an incidence algebra of a poset (not just quasi-ordered set), because of the upper triangular condition. Then $\eta(A)\subseteq B$, and $B$ naturally acts on $V$. Hence, $V$ is a $B$-module such that the restriction of $V$ to $A$ via $\eta$ gives back the old module ${}_AV$. Moreover, $V$ is thin as a $B$-module, by our results. We call such an $V$ a ``thin cover" of ${}_AV$. We record this formally in the following definition.   

\begin{definition}
Let $V$ be a finite dimensional module over a finite dimensional basic algebra $A$. A thin cover of $V$ is a pair $(W,B)$ where $B$ is an incidence algebra containing (extending) $A/\ann_A(V)$ and $W$ is a faithful thin $B$-module, and such that $V$ is the restriction of $W$ to $A$.
\end{definition}

We note that the collection of all posets (up to isomorphism) that can be obtained via this process is an invariant of the representation $V$, which is intimately related to the lattice of submodules of $V$. 

To apply this idea to the no-gap/accessibility conjecture, note that $V$ is an indecomposable module, and $B$ is as an incidence algebra as above containing $A/\ann(V)$, then $V$ remains indecomposable over $B$. By the combinatorial part, it has either a maximal submodule $M$ or maximal quotient $V/S$ which is indecomposable, and the restriction of this module to $A$ seems like a good candidate for the module conjectured to exist. These modules will not always restrict to indecomposable modules, but nevertheless, if such a submodule or quotient exists, it can be  obtained from an invariant flag, and consequently it will be one of the modules obtained from some thin cover. We refer to Example 4.8 of \cite{IS}, where it is shown how the indecomposable of maximal dimension of quivers of type ${\mathbb D}_4$  (and ${\mathbb D}_n$) can be obtained by restriction from quivers of type ${\mathbb A}_m$, and corresponding maximal submodules and quotients can be read from appropriate restrictions. 

\begin{example}
Consider the quiver 
$$
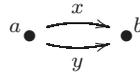
\xymatrix{{}^{a}\bullet \ar@/^1ex/[r]^{x}\ar@/_1ex/[r]_y & \bullet^b}$$
let $A$ be its path algebra (with multiplication following concatenation as before) and let $P$ be the $3$-dimensional projective left module with basis $\{x,y,b\}$. Consider the identification $\End_\KK(P)=M_3(\KK)$ via this basis. The module $P$ is faithful, and $A$ embeds in $\End_\KK(V)$ as the set of matrices $\left\{\left( \begin{array}{ccc} \alpha & 0 &  \beta \\ 0 & \alpha & \gamma \\ 0 & 0 & \delta \end{array}\right) \, \vert\, \alpha,\beta,\gamma,\delta\in\KK \right\}$. The smallest incidence algebra containing $A$ is $\left( \begin{array}{ccc} \KK & 0 & \KK \\ 0 & \KK & \KK \\ 0 & 0 & \KK  \end{array}\right)$, which is the incidence algebra of the poset $$
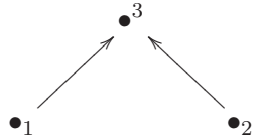
\xymatrix{& \bullet^3 & \\ \bullet_1\ar[ur] & & \bullet_2\ar[ul]}$$
It is not difficult to see that up to isomorohism this is the only thin cover of $P$. The second picture captures the ``generic" shape of the lattice of submodules of $P$. 
\end{example}

\begin{example}
Let $Q$ be the quiver of type a ${\mathbb D}_4$ wich has a source vertex. As noted before, \cite[Example 4.8]{IS} can be regarded as giving a thin cover of the maximal indecomposable $V$ of $Q$; the incidence algebra of the cover is the path algebra of an ${\mathbb A}_5$ quiver. We give here another example of a thin cover of this $V$. Let $P$ be the poset
$${\small
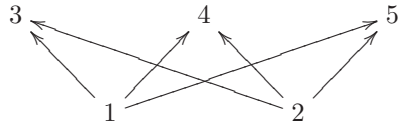
\xymatrix{ 3 & & 4 & & 5 \\ & 1\ar[ul]\ar[ur]\ar[urrr] & & 2\ar[ulll]\ar[ul]\ar[ur] & }}$$
and consider the embedding of algebras $$A=\left( \begin{array}{ccccc} \lambda & 0 & \mu & \delta & \theta \\ 0 & \lambda & \mu & \delta & \theta \\ 0 & 0 & \alpha & 0 & 0 \\ 0 & 0 & 0 & \beta & 0 \\ 0 & 0 & 0 & 0 & \gamma  \end{array} \right)
\hookrightarrow
\left( \begin{array}{ccccc} \KK & 0 & \KK & \KK & \KK \\ 0 & \KK & \KK & \KK & \KK \\ 0 & 0 & \KK & 0 & 0 \\ 0 & 0 & 0 & \KK & 0 \\ 0 & 0 & 0 & 0 & \KK  \end{array} \right)=I(P,\KK)=B$$
Here, in the in the algebra $A$, $\lambda,\mu,\delta,\theta,\alpha,\beta,\gamma$ are arbitrary elements of $\KK$. It is not difficult to note that $A$ is isomorphic to the path algebra of $Q$ and $I(P,\KK)$ is a thin cover; moreover, the restriction of the defining representation of $I(P,\KK)$ to $A$ is precisely the indecomposable $V$. Note also that the restriction to $A$ of the quotient of ${}_{B}V$ obtained by eliminating vertex $1$ remains indecomposable over $A$ (nevertheless, the restriction of any of its still indecomposable maximal $B$-submodules are no longer indecomposable over $A$). It is not difficult to see also that this thin cover is ``generic" in the appropriate affine space (consider all ways of defining $V$ up to isomorphism).     
\end{example}

The last example suggests that one may try to single out special thin covers by considering an appropriate Zariski topology, and looking at what thin covers are generic; again, \cite[Example 4.8]{IS} shows that even non-generic covers can be of use for the above mentioned accessibility conjecture.  



\section{Categorification}

We aim to give a few categorical interpretations and applications of distributive representations and incidence algebras. 

\subsection*{Deformations of monoidal structures of vector spaces graded by posets}

We begin with an interpretation of the third cohomology group $H^3(\Delta(P),\KK^*)$. This is will be  similar to the categories $\KK{-\mathbf{Vec}}_G^\omega$ \cite[Section 2.3]{EGNO}, of $G$-graded vector spaces over a group (or monoid) $G$ twisted by a 3-cocycle, equivalently, modules over the algebra of functions on $G$ viewed as {\it quasi-Hopf algebra} (comultiplication coassociative up to conjugation by $\omega^*$), or comodules over the group algebra $\KK G$ viewed as a {\it co-quasi-Hopf algebra} (with multiplication associative up to conjugation by $\omega$). In fact, if one extends these well known constructions to semigroup algebras of partially defined semigroups $S$ and $S$-graded vector spaces, then the two can be interpreted as instances of the same situation.

We refer the reader to \cite{DNR,EGNO} for definitions regarding coalgebras, bialgebras, monoidal and tensor categories. On the incidence algebra $I(P,\KK)$ define the comultiplication $\delta:I(P,\KK)\rightarrow I(P,\KK)\otimes I(P,\KK)$ by $\delta(f_{xy})=f_{xy}\otimes f_{xy}$. This is inherited from a correspondingly defined comultiplication on the path algebra. Namely, on the path algebra $\KK[Q]$ of any quiver $Q$ one can define the comultiplication $\delta$ by $\delta(p)=p \otimes p$ for all paths $p$. It is easy to note that with this $\KK[Q]$ becomes a weak-bialgebra (possibly non-unital if $Q$ has infinitely many vertices); this comultiplication is ``responsible" for the monoidal category structure on $Rep(Q)=\KK[Q]{\rm-Mod}$. Indeed, any two $Q$-representations $M,N$ can be tensored point-wise, i.e. $(M\otimes N)_x=M_x\otimes N_x$ and for each arrow $a:x\rightarrow y$, the morphism $(M\otimes N)_a$ corresponding to this arrow is $(M\otimes N)_a=M_a\otimes N_a$. This means that $a$ acts on $M\otimes N$ as $a\cdot (m\otimes n)=(a\cdot m)\otimes (a\cdot n)$, therefore, the comultiplication $\delta$ is obtained. 

Given any quiver $Q$ and an ideal $I$ of $\KK[Q]$, if $I$ is also a coideal, then the bialgebra structure of $\KK[Q]$ induces to $\KK[Q]/I$. This is the case for an incidence algebra $I(P,\KK)$, since the kernel $K$ of the canonical map $\KK[Q]\rightarrow I(P,\KK)$ (where $Q$ is the Hasse diagram of $P$) is generated by relations $p-q$ with $p,q$ paths sharing starting points, as well as ending points. This $K$ is easily seen to be a coideal. Alternatively, one sees immediately that such relations (i.e. commutativity of appropriate diagrams) are preserved by tensoring two representations of the poset $P$; hence tensoring will yield a new representation of the poset. 

Let $\omega\in H^3(\Delta(P),\KK^*)$. We will consider the category of $\KK$-vector spaces graded by $P$, or more precisely, by intervals in $P$. Hence, objects are $V=\bigoplus\limits_{x\leq y;\,x,y\in P}V_{x,y}$, with component preserving morphisms, and a tensor product defined by $(V\otimes W)_{x,y}=\bigoplus\limits_{x\leq z\leq y}V_{x,z}\otimes V_{z,y}$ (note that unlike tensor categories, in such representations-of-quivers categories one may well have tensor products of non-zero objects be zero). The definition is extended similarly on morphisms. The category has a unit object ${\mathbf 1}=\bigoplus\limits_{x\in P}\KK_x$, where $\KK_x$ is the 1-dimensional vector space ``concentrated" in degree $(x,x)$. Denote this category $\KK-{\mathbf{Vec}}_P$. This category can also be regarded as the category of comodules over $I(P,\KK)$, or equivalently, the category of modules over the incidence coalgebra $\KK P$ of $P$. This is a bialgebra which, as a coalgebra, is dual to $I(P,\KK)$; let $c_{xy}$ be the basis dual to $f_{xy}$, so $\delta_{\KK P}(c_{xy})=\sum\limits_{x\leq z\leq y}c_{xz}\otimes c_{zy}$. On this basis, the multiplication of $\KK P$ is defined such that $c_{xy}$ are orthogonal idempotents (dual to the comultiplication of $I(P,\KK)$), so $\KK P$ is commutative semisimple as an algebra. This last part is similar to regarding $\KK{-\mathbf{Vec}}_G$ as the category of modules over the function (Hopf) algebra of $G$.  

Finally, let $\KK-{\mathbf{Vec}_P^\omega}$ be this category as an abelian category, but with tensor product having the associativity constrain changed to $a_{x,y,z,t}=\omega(x,y,z,t):\KK f_{xy}\otimes (\KK f_{yz}\otimes \KK f_{zt})\longrightarrow (\KK f_{xy}\otimes \KK f_{yz})\otimes \KK f_{zt}$. By a computation similar to that in the case of group graded vector spaces and to the calculations proving Theorem \ref{t.clsdef}, it is straightforward to note that the pentagon diagram axiom (coherence of associativity constraint) is equivalent to the fact that $\omega$ is a 3-cocycle (note: the diagram has two paths, one with two arrows and one with three; these correspond to 5 terms and yield two terms of one sign and three of opposite, leading to the 3-cocycle condition), and that two 3-cocycles determine the same monoidal category if and only if they belong to the same orbit of the action of $\Aut(P)$ on $H^3(\Delta(P),\KK^*)$. 

\begin{corollary}
The categories $\KK-{\mathbf Vec}_P^\omega$ deforming the monoidal structure of $\KK-{\mathbf Vec}_P$ are classified up to equivalence by $H^3(\Delta(P),\KK^*)/\Aut(P)$ with the natural action of $\Aut(P)$. 
\end{corollary}

\subsection*{A special subring of $Rep(P)$ and the Grothendieck ring}

We start by illustrating the idea with an example. Consider the poset $\{b,a,c\}; b>a<c$, whose incidence algebra is simply the quiver algebra of an ${\mathbb{A}}_3$ quiver: $\xymatrix{\bullet & \bullet \ar[l]\ar[r] & \bullet}$. The (unique) representations $M,N$ of dimension vectors $(1,1,0)$ and $(0,1,1)$ respectively are distributive, and their tensor product is $M\otimes N=(1,1,0)\otimes (0,1,1)=(0,1,0)$, the unique representation of dimension vector $(0,1,0)$. In general, any distributive representation of a poset $P$ will consist of $0$ and $1$-dimensional spaces placed at vertices (since it is a thin representation). Let $M$ be such a representation with support poset $S$; then the 1-dimensional spaces appear at vertices in the support $S$ of $M$. If $M,N$ be two distributive representations of $P$, with support ${\rm Supp}(M),{\rm Supp}(N)$ and cocycles $\alpha(M),\alpha(N)$ respectively; it will be convenient consider the cocycle $\alpha$ of a distributive representation $M$ of support $S$ defined for all $x\leq y$ (not only for $x\leq_S y$) by extending the definition to $\alpha_{x,y}=0$ whenever $x\leq y$ but $x\not\leq_S y$. The following result, describing the monoidal structure of distributive representations is easy to see, having in mind the example above. We will need to define another combinatorial notion, the intersection of closed subposets of $P$: if $(S,\leq_S),(T,\leq_T)$ are two closed subposets of $P$, their {\it intersection or wedge} $(S\wedge T,\leq_{S\wedge T})$ is defined by $S\wedge T= S\cap T$ as sets and $x\leq_{S\wedge T} y$ if and only if $x\leq_S y$ and $x\leq_T y$. It is straightforward to check that the wedge of two closed sub-posets of $P$ is again a closed sub-poset.

\begin{theorem}\label{t.RRdis}
Let ${\mathcal D}=\bigsqcup\limits_{(S,\leq_S){\rm\,closed}}H^1(\Delta(S),\KK^*)$ be the set of distributive representations of $I(P,\KK)$, equivalently, the set of thin representations, parametrized as in Theorem \ref{t.discls}. Then ${\mathcal D}$ is a semigroup with respect to the tensor product of representations, described in terms of the parametrization as follows: ${\rm Supp}(M\otimes N)={\rm Supp}(M)\wedge {\rm Supp}(N)$ and $\alpha(M\otimes N)=\alpha(M)\alpha(N)$, where multiplication is done point-wise (in $\KK$); the non-zero part of $\alpha(M\otimes N)$ will be a 1-cocycle for ${\rm Supp}(M\otimes N)$. The semigroup algebra of this semigroup is a subring of the representation ring of $P$.
\end{theorem}

To illustrate, we go back to the example in the beginning of this subsection, with the poset $P:b>a<c$ and $M,N$ the representations of dimension vectors $(1,1,0)$ and $(0,1,1)$. We have ${\rm Supp}(M)=(\{a,b\}; a<b)$ and ${\rm Supp}(N)=(\{a,c\}, a<c)$ as posets. Let us write the ``extended" 1-cocycles defined on the set of ordered pairs (intervals) $C_1=C_1(P)=\{(x,y)|x,y,z\in\{a,b,c\}, x\leq y\}$ (the 1-skeleton of $\Delta(P)$), as ordered strings $\alpha(b,b);\alpha(a,b);\alpha(a,a);\alpha(a,c);\alpha(c,c)$. Then  $\alpha(M)=[1,1,1,0,0]$ and $\alpha(N)=[0,0,1,1,1]$. Furthermore, $M\otimes N$ is described by ${\rm Supp}(M\otimes N)=\{a\}={\rm Supp}(M)\cap {\rm Supp}(N)$ and cocycle $\alpha(M\otimes N)=[0,0,1,0,0]=\alpha(M)\alpha(N)$. 

This also describes a semigroup structure on the set of distributive modules ${\mathcal D}$ of $I(P,\KK)$, and at the same time ${\mathcal D}$ generates a tractable subring of the representation ring of $P$, which is in fact a semigroup ring. Note that it is also unital, with the defining representation of $P$ being a unit. 
 
\subsection*{The Grothendieck ring of a poset}



Consider the Grothendieck group $K_0(P)$, which is, by definition, the Grothendieck group $K_0$ of the category of $I(P,\KK)$-modules. Is the free abelian group with basis the (isomorphism types of) projective $I(P,\KK)$-modules. Of course, there is another important notion of Grothendieck group $G_0(P)$ often used for finite dimensional algebras, which is the Grothendieck group of the monoid freely generated with isomorphism classes of finitely generated modules, modulo relations given by extensions (in this case, $G_0(P)$ is simply the free abelian group with basis the isomorphism types of simple modules), but the usual K-theoretic group is important from the point of view of monoidal categories: in such cases $K_0$ is often a ring  (in fact, for the case of finite tensor categories, they both $K_0$ and $G_0$ have a ring structure and they are isomorphic \cite{EGNO}).


Note that the tensor product of simple $P$-representations is $S_x\otimes S_y=\delta_{xy} S_x$, and hence, $G_0(P)$ is simply a product of $\ZZ$'s as a ring. For a poset $P$, the (appropriate version of the) ring $K_0$ is more interesting.  

In general, for a poset $P$ it may be that $K_0(P)$ is not closed under multiplication. Consider the poset with Hasse diagram
$$
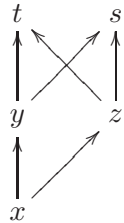
\xymatrix{t & s \\ y \ar[u]\ar[ur] & z\ar[ul]\ar[u]\\ x\ar[u]\ar[ur]}$$
so $P_t={\rm Span}_\KK\{f_{tt},f_{yt},f_{zt},f_{xt}\}$ and $P_s={\rm Span}_\KK\{f_{ss},f_{zs},f_{ys},f_{xs}\}$; the corresponding representations have dimension vectors, written in the order $(x,y,z,t,s)$, given by $(1,1,1,1,0)$ and $(1,1,1,0,1)$, and their tensor product is the unique representation which has dimension vector $(1,1,1,0,0)$; this representation is not projective. 

One could remedy this in two ways. One would be consider the subring of the representation ring consisting of all multiplicity free representations which are sub quotients of the defining representation. These can be regarded a the ``undeformed" thin representations of $P$. The defining representation has $\KK$ at all vertices and identity for of all morphims; a subquotient will replace some of these $\KK$'s as well as  some of the identity maps with $0$'s, such that the resulting representation (of the Hasse quiver of $P$) remains a representation of $P$ (i.e. the appropriate diagrams remain commutative). That means such representations are completely determined by their support (since the cocycle will always be taken as the trivial all 1-cocycle on the support and 0 elsewhere). Let $Cl(P)$ denote the set of closed subposets of $P$ in the sense of Definition \ref{d.closed}, and for $S=(S,\leq_S)\in Cl(P)$, let $R_S$ be the corresponding representation (it has $\KK$ on vertices corresponding to $x\in S$ and identity for all arrows $x\rightarrow y$ for which $x\leq_S y$). As noted, $Cl(P)$ is closed under the wedge of posets.

We have the following proposition, which now follows from the above.

\begin{proposition}
Let ${\mathcal U}(P)$ be the set of undeformed thin representations as above. Then $R_S\otimes R_T\cong R_{S\wedge T}$, for any $S,T\in Cl(P)$. Hence, ${\mathcal U}(P)$ is closed under tensor products, and it is a monoid isomorphic to $(Cl(P),\wedge)$. Furthermore, the semigroup algebra $\ZZ[Cl(P),\wedge]$ is a unital subring of the representation ring $Rep(P)$.
\end{proposition}

It is an interesting combinatorial problem to determine properties of such rings. For each subset $X$ of $P$, let $P_{\leq X}=\{t\in P| \, t\leq x,\,\forall\,x\in X\}$; obviously, $P_{\leq X}\in Cl(P)$ and $P_{\leq X}\wedge P_{\leq Y}=P_{\leq X\cup Y}$. In particular, the projective modules are $R_{P_{\leq a}}$, for $a\in P$. Hence, $\{P_{\leq X}| X\subseteq P\}$ provides a submonoid of ${\mathcal U}(P)$, and its span provides another interesting subring of $Rep(P)$. Note that, of course, some of the $P_{\leq X}$ may be equal, in general. 

In particular, let $P$ be a meet-semilattice. That means that $(P,\leq)$ is a poset such that every two elements $x,y$ in $P$ have an infimum $x\wedge y$ (meet) in $P$; we do not require that $P$ has a unique maximal element (sometimes this is called pseudo-semilattice; we refer to the ncatlab discussion on this). We note that meet-semilattices are precisely the commutative idempotent semigroups. Since the projective indecomposables are $P(x)=R_{P_{\leq x}}$, using the notations above, we obtain the following categorification result for join semilattices.

\begin{proposition}
Let $P$ be a meet-semilattice. Then $K_0(P)$ is closed under tensor products - $P(x)\otimes P(y)=P(x\wedge y)$ - and the Grothendieck ring $K_0(P)$ is isomorphic to the semigroup algebra $\ZZ[P,\wedge]$. By duality, any join-semilattice can be categorized similarly. 
\end{proposition}

This seems to suggest a general {\bf principle for categorification}. We do not attempt to give a rigorous formulation of this - the theme of categorification is of very active current interest and studied and developed by many authors (we only mention the survey \cite{KMS} and references therein); but note it as a possible avenue of investigation which may apply to certain types of situations. To categorify a certain algebraic structure $S$ (semigroup, etc.), first construct a combinatorial object ``modeled" on $S$, with ``points" the elements of $S$ and relations provided by the algebraic properties of (finite) subsets of $S$. Then consider the category of representations of such a combinatorial object, and its Grothendieck group as a candidate for categorifying (a certain part) of $S$. 

\subsection*{A Hall-algebra type of construction} We only briefly note that it is possible to consider a Hall-type of coalgebra, using again distributive representations: consider the set ${\mathcal D}$ of isomorphism classes of distributive representations, and a vector space $V$ with (formal) base $\mathcal D$. We can define a comultiplication $\Delta$ on $V$ of the following type: for a distributive representation $M$, one can define $\Delta([M])=\sum\limits_{X<M}(\gamma_{X,M})[X]\otimes [M/X]$ where $\gamma_{X,M}$ could be any coefficients which make the structure coassociative. It would be perhaps interesting to study such coalgebras and their dual algebras; they would be close to incidence algebras generated by closed subsets of $P$ and their order inclusion. This poset is slightly more general than the face poset of $P$ (with elements the simplices in $P$ ordered by inclusion). 

\subsection*{Questions}

We end by noting a few questions or problems that seem to arise from the current study.

{\bf Question 1.}
Study the Ringel-Hall type of algebra presented in the last paragraph.

{\bf Question 2.}
Extend the results presented here to the infinite dimensional setting, as noted in the introductory remarks.

{\bf Question 3.}
Give a complete description of the automorphism group of a deformation of an incidence algebra, possibly as an iterated semidirect product; this should also include a description of the group of units of $I_\lambda(P,\KK)$, up to isomorphism.  

{\bf Question 4.} 
Give a complete description of the Lie algebra of derivations of a (deformation) of an incidence algebra, in terms of inner derivations and outer derivations ($HH^1$).

{\bf Question 5.}
In the context of finite dimensional basic pointed algebras, find the proper combinatorial setup, and associated topological gadget, to describe left/right semidistributive algebras (i.e. when projective indecomposable are distributive in general), respectively, algebras with finitely many ideals. The question can restrict to the acyclic case, or can be considered in full generality.


{\bf Question 6.}
What types of Grothendieck rings can be obtained by possibly using more general combinatorial algebras as in Question 5? For particular examples of posets, compute explicitly the rings generated by the thin representations inside the representation ring. Find consequences on and study the representation ring of a poset $P$ using these subrings (note that in particular there are some ``large" group rings inside $Rep(P)$ when $P$ has $H_1$ of free rank $\geq 1$). 

{\bf Question 7.}
We end by reiterating the accessibility conjecture of Ringel and Bongartz:\\
If $M$ is an indecomposable module of finite length over some algebra $A$, is $M$ accessible, at least in the case when $M$ is distributive? Is there a reduction procedure to this case? \\
It is tempting (and perhaps within reach) to try to extend the ideas in the proof of Theorem \ref{t.accessible} for general distributive modules. We believe any partial answer to this would be interesting.

\bigskip

\begin{center}
{\sc Acknowledgment}
\end{center}
MI would like to very gratefully thank Gene Abrams, Marcelo Aguiar, Ibrahim Assem, Calin Chindris, Birge Huisgen-Zimmermann, Kiyoshi Igusa, Alex Martsinkovsky, Marcel Wild, Dan Zacharia, as well as to current or former colleagues Vic Camillo, Alexander Sistko, David Meyer for many interesting discussions on various related topics and observations on several previous versions of the paper, which greatly helped with the development of the ideas.  







\end{document}